\documentclass{amsart}

\usepackage{amssymb}
\usepackage{amsmath}
\usepackage{amsthm}
\usepackage{amsfonts}

\usepackage{a4wide}
\usepackage{longtable}

\usepackage{multirow}
\usepackage{graphicx}

\newtheorem{proposition}{Proposition}[section]
\newtheorem{theorem}{Theorem}[section]
\newtheorem{corollary}{Corollary}[theorem]

\theoremstyle{definition}
\newtheorem{definition}{Definition}[section]
\newtheorem{remark}{Remark}[section]
\newtheorem{example}{Example}[section]

\begin{document}

\title[The second \(p\)-class group]
{The second \(p\)-class group of a number field}

\author{Daniel C. Mayer}
\address{Naglergasse 53\\8010 Graz\\Austria}
\email{algebraic.number.theory@algebra.at}
\urladdr{http://www.algebra.at}
\thanks{Research supported by the
Austrian Science Fund,
Grant Nr. J0497-PHY}

\subjclass[2000]{Primary 11R29,11R11,11R16,11R20; Secondary 20D15,20F12,20F14}
\keywords{\(p\)-class groups, principalisation, \(p\)-class fields, quadratic fields, cubic fields,
dihedral fields, metabelian \(p\)-groups, commutator calculus, central series, coclass graphs}

\date{December 24, 2010}

\dedicatory{Dedicated to the memory of Arnold Scholz}


\begin{abstract}
For a prime \(p\ge 2\) and
a number field \(K\) with \(p\)-class group of type \((p,p)\)
it is shown that
the class, coclass, and further invariants of the metabelian Galois group
\(G=\mathrm{Gal}(\mathrm{F}_p^2(K)\vert K)\)
of the second Hilbert \(p\)-class field \(\mathrm{F}_p^2(K)\) of \(K\)
are determined by
the \(p\)-class numbers of the unramified cyclic extensions
\(N_i\vert K\), \(1\le i\le p+1\), of relative degree \(p\).
In the case of a quadratic field \(K=\mathbb{Q}(\sqrt{D})\)
and an odd prime \(p\ge 3\),
the invariants of \(G\) are derived from
the \(p\)-class numbers
of the non-Galois subfields \(L_i\vert\mathbb{Q}\) of absolute degree \(p\)
of the dihedral fields \(N_i\).
As an application, the structure of the automorphism group
\(G=\mathrm{Gal}(\mathrm{F}_3^2(K)\vert K)\)
of the second Hilbert \(3\)-class field \(\mathrm{F}_3^2(K)\) is analysed
for all quadratic fields \(K\) with discriminant \(-10^6<D<10^7\)
and \(3\)-class group of type \((3,3)\)
by computing their principalisation types.
The distribution of these metabelian \(3\)-groups \(G\)
on the coclass graphs \(\mathcal{G}(3,r)\), \(1\le r\le 6\),
in the sense of Eick and Leedham-Green is investigated.
\end{abstract}

\maketitle


\section{Introduction}
\label{s:Intro}

For an algebraic number field \(K\) and
a prime \(p\ge 2\) we denote by
\(\mathrm{Cl}_p(K)\)
the \(p\)-class group of \(K\),
that is the Sylow \(p\)-subgroup
\(\mathrm{Syl}_p\mathrm{Cl}(K)\)
of its ideal class group.
In this paper we shall be concerned with
number fields having an
elementary abelian bicyclic
\(p\)-class group of type \((p,p)\).
We define the \(p\)-class field tower of \(K\),
\(K<\mathrm{F}_p^1(K)\le\mathrm{F}_p^2(K)\le\ldots\le\mathrm{F}_p^n(K)\le\ldots\),
recursively by \(\mathrm{F}_p^0(K)=K\) and
\(\mathrm{F}_p^n(K)=\mathrm{F}_p^1\left(\mathrm{F}_p^{n-1}(K)\right)\)
for \(n\ge 1\), where each successor is the
Hilbert \(p\)-class field of its predecessor,
that is the maximal abelian unramified extension
with a power of \(p\) as relative degree.
The \(n\)th Hilbert \(p\)-class field \(\mathrm{F}_p^n(K)\) of \(K\)
is an unramified Galois extension of \(K\)
for each \(n\ge 1\) (cfr. Hasse \cite[p.164, \S 27]{Ha2}).
It is non-abelian for \(n\ge 2\),
except in the degenerate case \(\mathrm{F}_p^2(K)=\mathrm{F}_p^1(K)\)
of a single-stage tower.

According to the reciprocity law of Artin \cite{Ar1,Ar2},
the Galois group of the first Hilbert \(p\)-class field of \(K\),
\(\mathrm{Gal}(\mathrm{F}_p^1(K)\vert K)\simeq\mathrm{Cl}_p(K)\),
is isomorphic to the \(p\)-class group of \(K\) (cfr. Miyake \cite{My}).
Therefore the automorphism group
\(G=\mathrm{Gal}(\mathrm{F}_p^n(K)\vert K)\)
of the \(n\)-stage tower
is called the \(n\)th \(p\)-class group of \(K\).
Since \(\mathrm{F}_p^1(K)\) is maximal
among all abelian unramified \(p\)-extensions of \(K\),
the commutator subgroup of \(G\) is given by
\(\gamma_2(G)=[G,G]=G^\prime
=\mathrm{Gal}(\mathrm{F}_p^n(K)\vert\mathrm{F}_p^1(K))\)
and the abelianisation
\(G/\gamma_2(G)\simeq\mathrm{Gal}(\mathrm{F}_p^1(K)\vert K)\simeq\mathrm{Cl}_p(K)\)
is isomorphic to the \(p\)-class group of \(K\).

The aim of the present paper is to investigate the
\textit{second \(p\)-class group} \(G=\mathrm{Gal}(\mathrm{F}_p^2(K)\vert K)\),
that is the Galois group of the \textit{two-stage tower},
\(K<\mathrm{F}_p^1(K)\le\mathrm{F}_p^2(K)\),
of a base field \(K\) with \(p\)-class group of type \((p,p)\).
The second \(p\)-class group \(G\) is distinguished by the special property that its commutator subgroup
\(\gamma_2(G)=\mathrm{Gal}(\mathrm{F}_p^2(K)\vert\mathrm{F}_p^1(K))\simeq\mathrm{Cl}_p(\mathrm{F}_p^1(K))\)
is isomorphic to the \(p\)-class group of the first Hilbert \(p\)-class field of \(K\),
whence \(G\) is metabelian.

The theory of second \(p\)-class groups was initiated by Scholz and Taussky \cite{So,SoTa},
using Schreier's concept of group extensions \cite{Sr1,Sr2}.
It was continued by Heider and Schmithals \cite{HeSm}
with the aid of Hall's isoclinism families \cite{Hl} and presentations by Blackburn \cite{Bl}.
Here we shall connect it with the theory of coclass graphs by Newman and Leedham-Green \cite{LgNm},
which has been developed by Ascione \cite{AHL,As}, Nebelung \cite{Ne}, O'Brien \cite{NmOb}, McKay \cite{LgMk},
Eick, et al. \cite{EkLg,DEF,EkFs}.

After the definition of three fundamental isomorphism invariants \(s,e,k\) of metabelian \(p\)-groups
in section \ref{s:IsoInv}, we prove in section \ref{s:Dir} that
the class, coclass, and further invariants of the second \(p\)-class group \(G\) of \(K\)
determine the \(p\)-class numbers \(\mathrm{h}_p(N_i)\)
of the \(p+1\) unramified cyclic extensions \(N_i\) of \(K\) of relative degree \(p\),
because the \(p\)-class group \(\mathrm{Cl}_p(N_i)\)
is isomorphic to the commutator factor group of
the maximal normal subgroup \(M_i=\mathrm{Gal}(\mathrm{F}_p^2(K)\vert N_i)\)
of \(G\), for each \(1\le i\le p+1\), by Artin's reciprocity law.
The investigation is separated into two parts.
In subsection \ref{ss:Max} we treat the metabelian \(p\)-groups \(G\) of coclass \(\mathrm{cc}(G)=1\),
for which generators and relations have been given for an arbitrary prime \(p\ge 2\)
by Blackburn \cite{Bl} and more generally by Miech \cite{Mi}.
Subsection \ref{ss:Low} deals with metabelian \(3\)-groups \(G\) of coclass \(\mathrm{cc}(G)\ge 2\)
with abelianisation \(G/\gamma_2(G)\) of type \((3,3)\),
whose presentations are due to Nebelung \cite{Ne}.

In section \ref{s:Qdr} we use the theory of dihedral field extensions,
in particular some well-known class number relations by Scholz and Moser \cite{So,Mo},
to show that in the case of a quadratic base field \(K=\mathbb{Q}(\sqrt{D})\)
and an odd prime \(p\ge 3\)
the invariants of the second \(p\)-class group \(G\) determine
the \(p\)-class numbers \(\mathrm{h}_p(N_i)\) of the dihedral fields \(N_i\)
and the \(p\)-class numbers \(\mathrm{h}_p(L_i)\)
of the non-Galois subfields \(L_i\) of absolute degree \(p\) of the \(N_i\).
In contrast to \cite{Ma2}, where we have solved the multiplicity problem for discriminants
of dihedral fields which are ramified with conductor \(f>1\) over their quadratic subfield,
we are now concerned exclusively with unramified extensions having \(f=1\).
For a real quadratic base field \(K\),
the cohomology
\(\mathrm{H}^0(\mathrm{Gal}(N_i\vert K),U(N_i))\)
of the unit groups \(U(N_i)\)
with respect to the cyclic Galois groups \(\mathrm{Gal}(N_i\vert K)\)
must be considered
to distinguish between partial and total principalisation of \(K\) in the \(N_i\).

In the sections
\ref{s:InvThmCub}, \ref{s:InvThmQnt}, and \ref{s:InvThmQdr}
we develop new algorithmic methods to compute, conversely,
the class \(\mathrm{cl}(G)\), coclass \(\mathrm{cc}(G)\),
and the invariants \(s,e,k\) of the second \(p\)-class group \(G\)
from the \(p\)-class numbers \(\mathrm{h}_p(L_i)\)
or \(\mathrm{h}_p(N_i)\), for \(p=3\), \(p\ge 5\), and \(p=2\).

With the aid of these new methods
we have determined
the structure of the second \(3\)-class group
\(\mathrm{Gal}(\mathrm{F}_3^2(K)\vert K)\)
of the \(4\,596\) quadratic number fields \(K\) with discriminant \(-10^6<D<10^7\)
and \(3\)-class group of type \((3,3)\)
from the \(3\)-class numbers of the absolute cubic fields \(L_i\)
and the principalisation types \cite[3.3]{Ma3} of \(K\) in the relative cubic fields \(N_i\).
The results of these extensive computations,
which reveal sound statistical insight for the first time,
are presented in section \ref{s:NumRstCub}.
The distribution of the occurring metabelian \(3\)-groups
on the coclass graphs \(\mathcal{G}(3,r)\) with \(1\le r\le 6\)
in the sense of Eick, Leedham-Green, et al. \cite{EkLg,DEF,EkFs}
is investigated here and in the related paper \cite{Ma4}.

In section \ref{s:NumRstQdr} we present examples of second \(2\)-class groups
\(\mathrm{Gal}(\mathrm{F}_2^2(K)\vert K)\)
of increasing order,
corresponding to excited states of three principalisation types,
for complex quadratic fields \(K\) with
\(2\)-class group of type \((2,2)\),
extending the results for the ground states by H. Kisilevsky \cite{Ki}.


\section{The invariants \(s,e,k\) of metabelian \(p\)-groups}
\label{s:IsoInv}

To be able to express the main theorems concerning the \(p\)-class numbers
of the unramified cyclic extensions \(N_i\) of relative degree \(p\)
of a base field \(K\) with \(p\)-class group of type \((p,p)\),
we have to recall some fundamental concepts with respect to the metabelian
second \(p\)-class group \(G\) of \(K\).

Since the commutator factor group \(G/\gamma_2(G)\) of \(G\) is of type \((p,p)\),
the subgroup \(G^p\) of \(G\) generated by the \(p\)th powers
is contained in the commutator group \(\gamma_2(G)\),
which therefore coincides with the Frattini subgroup
\(\Phi(G)=\cap_{i=1}^{p+1}\,M_i=G^p\gamma_2(G)=\gamma_2(G)\).
According to the basis theorem of Burnside \cite[p.29, Th.1.12]{Bv},
the group \(G=\langle x,y\rangle\)
can thus be generated by two elements.

Like any finite \(p\)-group, \(G\) is nilpotent.
If we declare the lower central series of \(G\)
recursively by \(\gamma_1(G)=G\) and
 \(\gamma_j(G)=\lbrack\gamma_{j-1}(G),G\rbrack\) for \(j\ge 2\),
then we have Kaloujnine's commutator relations
\(\lbrack\gamma_j(G),\gamma_\ell(G)\rbrack\le\gamma_{j+\ell}(G)\) for \(j,\ell\ge 1\)
\cite[p.47, Cor.2]{Bl}
and for a certain index of nilpotence \(m\ge 2\) the series
\(\gamma_1(G)>\gamma_2(G)>\ldots>\gamma_{m-1}(G)>\gamma_m(G)=1\) becomes stationary.
The number of non-trivial factors \(\gamma_j(G)/\gamma_{j+1}(G)\)
is called the class \(\mathrm{cl}(G)=m-1\) of \(G\).

\(G\) is of maximal class
or a CF-group (cyclic factors), if and only if
\(G\) has the order \(\lvert G\rvert=p^n\)
with the index of nilpotence \(n=m\) as exponent,
that is, if all factors
\(\gamma_j(G)/\gamma_{j+1}(G)\) with \(2\le j\le m-1\)
are cyclic of order \(p\).
In this case, \(G\) is of coclass \(\mathrm{cc}(G)=n-\mathrm{cl}(G)=n-m+1=1\).

The centralisers
\(\chi_j(G)=\lbrace g\in G\mid\lbrack g,u\rbrack\in\gamma_{j+2}(G)\text{ for all }u\in\gamma_j(G)\rbrace\),
\(2\le j\le m-1\),
of two-step factor groups \(\gamma_j(G)/\gamma_{j+2}(G)\)
of the lower central series \cite[p.54, Lem.2.5]{Bl},
\(\chi_j(G)/\gamma_{j+2}(G)
=\mathrm{centraliser}_{G/\gamma_{j+2}(G)}(\gamma_j(G)/\gamma_{j+2}(G)),\)
that is, the biggest subgroups of \(G\) with
\(\lbrack\chi_j(G),\gamma_j(G)\rbrack\le\gamma_{j+2}(G)\),
form an ascending chain of characteristic subgroups of \(G\),
\(\gamma_2(G)\le\chi_2(G)\le\ldots\le\chi_{m-2}(G)<\chi_{m-1}(G)=G,\)
which contain the commutator group \(\gamma_2(G)\).
\(\chi_j(G)\) coincides with \(G\), if and only if \(j\ge m-1\).

With an invariant \(s=s(G)\) we characterise
the first two-step centraliser,
which is strictly bigger than the commutator group \cite{Ne}.

\begin{definition}
\label{d:InvS}
Let \(2\le s\le m-1\) be defined by
\(s=\min\lbrace 2\le j\le m-1\mid\chi_j(G)>\gamma_2(G)\rbrace.\)
\end{definition}

The group \(\chi_s(G)\) coincides
either with one of the maximal subgroups \(M_i\)
or with \(G\).
In the first case, that is for \(s\le m-2\), we have
\(\gamma_2(G)=\chi_2(G)=\ldots=\chi_{s-1}(G)<\chi_s(G)=\ldots=\chi_{m-2}(G)<\chi_{m-1}(G)=G\).
The smallest value \(s=2\) occurs, if and only if \(G\) is of coclass \(\mathrm{cc}(G)=1\).

With a further invariant \(e=e(G)\) it will be expressed,
which factor \(\gamma_j(G)/\gamma_{j+1}(G)\) of the lower central series
is cyclic for the first time \cite{Ne}.

\begin{definition}
\label{d:InvE}
Let \(2\le e\le m-1\) with
\(e+1=\min\lbrace 3\le j\le m\mid 1\le\lvert\gamma_j(G)/\gamma_{j+1}(G)\rvert\le p\rbrace.\)
\end{definition}

In this definition we exclude the factor
\(\gamma_2(G)/\gamma_3(G)\), which is always cyclic.
The value \(e=2\) is characteristic for a group \(G\) of coclass \(\mathrm{cc}(G)=1\).
For \(p=3\) we have \(e=n-m+2=\mathrm{cc}(G)+1\) and \(e=s\),
except if \(e=m-2<s=m-1\) in the case of a cyclic center \(\zeta_1(G)\).
For \(e\ge 3\), that is for \(G\) of coclass \(\mathrm{cc}(G)\ge 2\), we can also define
\(e=\max\lbrace 3\le j\le m-1\mid\lvert\gamma_j(G)/\gamma_{j+1}(G)\rvert>p\rbrace\).
Then \(\gamma_e(G)/\gamma_{e+1}(G)\)
must be at least of \(p\)-rank two,
by \cite[p.49, Th.1.5]{Bl}.

Finally, the invariant \(k=k(G)\) is a measure for the deviation from the
greatest possible commutativity of the groups \(\chi_s(G)\) and \(\gamma_e(G)\)
\cite{Mi}.

\begin{definition}
\label{d:InvK}
Let \(0\le k\le m-e-1\) be defined by
\(\lbrack\chi_s(G),\gamma_e(G)\rbrack=\gamma_{m-k}(G).\)
\end{definition}


\section{Class numbers of the unramified cyclic extensions of relative degree \(p\)}
\label{s:Dir}


\subsection{Metabelian \(p\)-groups \(G\) of coclass \(\mathrm{cc}(G)=1\)}
\label{ss:Max}

We begin with a purely group theoretic statement
concerning the order of the commutator factor groups \(M_i/\gamma_2(M_i)\)
of the maximal normal subgroups \(M_i\)
of a metabelian \(p\)-group \(G\) of maximal class.
As a special case,
the abelian \(p\)-group \(G\) of type \((p,p)\) is included.

\begin{theorem}
\label{t:MaxCfg}

With a prime \(p\ge 2\),
let \(G\) be a \(p\)-group
of order \(\lvert G\rvert=p^m\) and
class \(\mathrm{cl}(G)=m-1\), where \(m\ge 2\).
Suppose that the commutator group \(\gamma_2(G)\) is abelian and
the commutator factor group \(G/\gamma_2(G)\) is of type \((p,p)\).
Let generators of \(G=\langle x,y\rangle\) be selected such that
\(x\in G\setminus\chi_2(G)\), if \(m\ge 4\), and \(y\in\chi_2(G)\setminus\gamma_2(G)\).
Assume that the order of the maximal normal subgroups
\(M_i=\langle g_i,\gamma_2(G)\rangle\)
is defined by \(g_1=y\) and \(g_i=xy^{i-2}\) for \(2\le i\le p+1\).
Finally, let the invariant \(k\) of \(G\) be declared by
\(\lbrack\chi_2(G),\gamma_2(G)\rbrack=\gamma_{m-k}(G)\),
where \(k=0\) for \(m\le 3\), \(0\le k\le m-4\) for \(m\ge 4\),
and \(0\le k\le\min\lbrace m-4,p-2\rbrace\) for \(m\ge p+1\) \cite{Mi}.

Then the order of the commutator factor groups of \(M_1,\ldots,M_{p+1}\) is given by

\begin{enumerate}
\item
\(\lvert M_i/\gamma_2(M_i)\rvert=p\) for \(1\le i\le p+1\), if \(m=2\),
\item
\(\lvert M_i/\gamma_2(M_i)\rvert=p^2\) for \(2\le i\le p+1\), if \(m\ge 3\),
\item
\(\lvert M_1/\gamma_2(M_1)\rvert=p^{m-k-1}\), if \(m\ge 3\).
\end{enumerate}

\end{theorem}

\begin{proof}

We define the main commutator
\(s_2=\lbrack y,x\rbrack\in\gamma_2(G)\)
and the higher commutators
\(s_j=\lbrack s_{j-1},x\rbrack=s_{j-1}^{x-1}\in\gamma_j(G)\) for \(j\ge 3\).
Then the nilpotence of \(G\) is expressed by the relation \(s_m=1\).
Generators and the order of the members of the lower central series are given by
\(\gamma_j(G)=\langle s_j,\ldots,s_{m-1}\rangle\)
\cite[p.58, Lem.2.9]{Bl} and
\(\lvert\gamma_j(G)\rvert=p^{m-j}\)
\cite[p.53, Lem.2.3]{Bl}
for \(2\le j\le m-1\).

The commutator group \(\gamma_2(G)\) is contained in each maximal subgroup
\(M_i=\langle g_i,\gamma_2(G)\rangle\)
as a normal subgroup of index \(p\).
According to \cite[p.52, Lem.2.1]{Bl} it follows that
\(\gamma_2(M_i)=\lbrack M_i,M_i\rbrack=\lbrack\gamma_2(G),M_i\rbrack
=\gamma_2(G)^{g_i-1}\).

\begin{enumerate}

\item
In the case \(m=2\), we have \(\mathrm{cl}(G)=1\),
whence \(G\) and all \(M_i\) are abelian.
Consequently \(\gamma_2(M_i)=1\) and
\(\lvert M_i/\gamma_2(M_i)\rvert=\lvert M_i\rvert=p^{m-1}=p\) for \(1\le i\le p+1\).

\item
In the case \(m=3\), \(G\) is an extra special group with
\(\gamma_2(G)=\Phi(G)=\zeta_1(G)\) cyclic of order \(p\).
Since \(s_2=s_{m-1}\in\zeta_1(G)\),
we obtain for \(2\le i\le p+1\), and also for \(i=1\), that
\[\gamma_2(M_i)=\gamma_2(G)^{g_i-1}=\langle s_2\rangle^{g_i-1}
=\langle\lbrack s_2,g_i\rbrack\rangle=1\]
and thus \(\lvert M_i/\gamma_2(M_i)\rvert=\lvert M_i\rvert=p^{m-1}=p^2\).\\
In the case \(m\ge 4\), we have
\[\gamma_2(G)=\chi_1(G)<\chi_2(G)=\ldots=\chi_{m-2}(G)<\chi_{m-1}(G)=G.\]
The elements \(g_i=xy^{i-2}\) with \(2\le i\le p+1\)
are contained in \(G\setminus\chi_2(G)\),
whence
\[\gamma_2(M_i)=\gamma_2(G)^{g_i-1}=\langle s_2,\ldots,s_{m-1}\rangle^{g_i-1}
=\langle\lbrack s_2,g_i\rbrack,\ldots,\lbrack s_{m-2},g_i\rbrack\rangle
\not\leq\gamma_4(G).\]
Since \(g_i\in G\setminus\chi_j(G)\), it follows that
\(\lbrack s_j,g_i\rbrack\in\gamma_{j+1}(G)\setminus\gamma_{j+2}(G)\)
for each \(2\le j\le m-2\).
Starting with
\(\langle\lbrack s_{m-2},g_i\rbrack\rangle=\langle s_{m-1}\rangle=\gamma_{m-1}(G)\),
we have by descending induction for \(m-3\ge j\ge 2\) that
\[\langle\lbrack s_j,g_i\rbrack,\lbrack s_{j+1},g_i\rbrack,\ldots,\lbrack s_{m-2},g_i\rbrack\rangle
=\langle\lbrack s_j,g_i\rbrack,\gamma_{j+2}(G)\rangle
=\langle s_{j+1},\gamma_{j+2}(G)\rangle=\gamma_{j+1}(G).\]
Therefore \(\gamma_2(M_i)=\langle s_3,\ldots,s_{m-1}\rangle=\gamma_3(G)\) and\\
\(\lvert M_i/\gamma_2(M_i)\rvert=\lvert M_i\rvert/\lvert\gamma_3(G)\rvert=p^{m-1}/p^{m-3}=p^2\)
for \(2\le i\le p+1\).

\item
The case \(m=3\) has been treated in the preceding paragraph.
In the case \(m\ge 4\), we distinguish the values \(k=0\) and \(k\ge 1\).\\
For \(k=0\) we have
\(s_2^{y-1}=\lbrack s_2,y\rbrack\in\lbrack\chi_2(G),\gamma_2(G)\rbrack=1\) and thus also
\[\lbrack s_j,y\rbrack=s_j^{y-1}=s_2^{(x-1)^{j-2}(y-1)}=s_2^{(y-1)(x-1)^{j-2}}=1\]
for all \(3\le j\le m-1\).\\
For \(k\ge 1\) we use, according to Miech \cite[p.332, Th.2, (2)]{Mi}, the existence of exponents\\
\(0\le a(m-k),\ldots,a(m-1)\le p-1\) with \(a(m-k)>0\), such that
\[s_2^{y-1}=\lbrack s_2,y\rbrack=\prod_{\ell=1}^k\,s_{m-\ell}^{a(m-\ell)}
\in\lbrack\chi_2(G),\gamma_2(G)\rbrack=\gamma_{m-k}(G).\]
Then the higher commutators with \(3\le j\le m-2\) are
\[\lbrack s_j,y\rbrack=s_j^{y-1}=s_2^{(x-1)^{j-2}(y-1)}=s_2^{(y-1)(x-1)^{j-2}}\]
\[=\left(\prod_{\ell=1}^k\,s_{m-\ell}^{a(m-\ell)}\right)^{(x-1)^{j-2}}
=\prod_{\ell=1}^k\,\left(s_{m-\ell}^{(x-1)^{j-2}}\right)^{a(m-\ell)}
=\prod_{\ell=1}^k\,s_{m+j-\ell-2}^{a(m-\ell)}
=\prod_{\ell=j-1}^k\,s_{m+j-\ell-2}^{a(m-\ell)},\]
since \(s_{m+j-\ell-2}=1\) for \(j-\ell-2\ge 0\), that is \(\ell\le j-2\).
Thus we have
\[\gamma_2(M_1)=\gamma_2(G)^{y-1}=\langle s_2,\ldots,s_{m-1}\rangle^{y-1}\]
\[=\langle\lbrack s_2,y\rbrack,\ldots,\lbrack s_{k+1},y\rbrack,\lbrack s_{k+2},y\rbrack,\ldots,\lbrack s_{m-1},y\rbrack\rangle
=\langle\lbrack s_2,y\rbrack,\ldots,\lbrack s_{k+1},y\rbrack\rangle,\]
since \(\lbrack s_j,y\rbrack=1\) for \(k+2\le j\le m-1\).
Observing that \(a(m-k)>0\) and starting with
\(\langle\lbrack s_{k+1},y\rbrack\rangle=\langle s_{m-1}^{a(m-k)}\rangle=\gamma_{m-1}(G)\),
we obtain by descending induction for \(k\ge j\ge 2\) that
\[\langle\lbrack s_j,y\rbrack,\lbrack s_{j+1},y\rbrack,\ldots,\lbrack s_{k+1},y\rbrack\rangle
=\langle\lbrack s_j,y\rbrack,\gamma_{m+j-k-1}(G)\rangle\]
\[=\langle s_{m+j-k-2}^{a(m-k)},\gamma_{m+j-k-1}(G)\rangle=\gamma_{m+j-k-2}(G).\]
Therefore \(\gamma_2(M_1)=\langle s_{m-k},\ldots,s_{m-1}\rangle=\gamma_{m-k}(G)\)
and\\
\(\lvert M_1/\gamma_2(M_1)\rvert=\lvert M_1\rvert/\lvert\gamma_{m-k}(G)\rvert
=p^{m-1}/p^{m-(m-k)}=p^{m-1}/p^k=p^{m-k-1}\).

\end{enumerate}
\end{proof}

\begin{corollary}
\label{c:MaxCg}

With the assumptions of theorem \ref{t:MaxCfg},
the commutator groups
of the maximal subgroups \(M_1,\ldots,M_{p+1}\)
of a metabelian \(p\)-group \(G\) of coclass \(\mathrm{cc}(G)=1\)
with order \(\lvert G\rvert=p^m\), \(m\ge 3\)
and \(\lbrack\chi_2(G),\gamma_2(G)\rbrack=\gamma_{m-k}(G)\), \(k\ge 0\)
are given by

\[\begin{array}{rclcrclc}
\gamma_2(M_1)&=&\gamma_{m-k}(G)&\text{ with }&\lvert\gamma_2(M_1)\rvert&=&p^k,&\\
\gamma_2(M_i)&=&\gamma_3(G)    &\text{ with }&\lvert\gamma_2(M_i)\rvert&=&p^{m-3}&\text{ for }2\le i\le p+1.
\end{array}\]

\noindent
\(M_1\) is abelian, if and only if \(k=0\), that is
\(\lbrack\chi_2(G),\gamma_2(G)\rbrack=1\).\\
\(M_2,\ldots,M_{p+1}\) are abelian, if and only if \(m=3\).

\end{corollary}

Now we reduce the desired determination of the \(p\)-class numbers
\(\mathrm{h}_p(N_i)\)
to the purely group theoretic preparations in theorem \ref{t:MaxCfg},
using Artin's reciprocity law \cite{Ar1}.

\begin{theorem}
\label{t:MaxDir}

Let \(K\) be an arbitrary base field
with \(p\)-class group \(\mathrm{Cl}_p(K)\) of type \((p,p)\).
Suppose that the second \(p\)-class group
\(G=\mathrm{Gal}(\mathrm{F}_p^2(K)\vert K)\)
is abelian or metabelian of coclass \(\mathrm{cc}(G)=1\)
with order \(\lvert G\rvert=p^m\)
and class \(\mathrm{cl}(G)=m-1\),
where \(m\ge 2\).
Under the assumptions of theorem \ref{t:MaxCfg}
for the generators of \(G\),
the \(p\)-class numbers of the
\(p+1\) unramified cyclic extension fields \(N_1,\ldots,N_{p+1}\) of \(K\)
of relative prime degree \(p\ge 2\) are given by

\begin{eqnarray*}
\mathrm{h}_p(N_i)&=&p\text{ for }1\le i\le p+1,\text{ if }m=2,\\
\mathrm{h}_p(N_i)&=&p^2\text{ for }2\le i\le p+1,\text{ if }m\ge 3,\\
\mathrm{h}_p(N_1)&=&
\begin{cases}
p^{m-1},&\text{ if }\lbrack\chi_2(G),\gamma_2(G)\rbrack=1,\ k=0,\ m\ge 3,\ p\ge 2,\\
p^{m-2},&\text{ if }\lbrack\chi_2(G),\gamma_2(G)\rbrack=\gamma_{m-1}(G),\ k=1,\ m\ge 5,\ p\ge 3,\\
p^{m-k-1},&\text{ if }\lbrack\chi_2(G),\gamma_2(G)\rbrack=\gamma_{m-k}(G),\ k\ge 2,\ m\ge 6,\ p\ge 5.\\
\end{cases}
\end{eqnarray*}

\end{theorem}

\begin{proof}

Due to the Artin isomorphism \cite[p.361, Allgemeines Reziprozit\"atsgesetz]{Ar1}
\[\mathrm{Cl}_p(N_i)\simeq\mathrm{Gal}(\mathrm{F}_p^1(N_i)|N_i),\]
the \(p\)-class group of the extension field \(N_i\) is isomorphic to the
relative Galois group of the Hilbert \(p\)-class field of \(N_i\) over \(N_i\)
for \(1\le i\le p+1\) (see also Miyake \cite[p.297, Cor.]{My}).
Galois theory yields a further isomorphism
\[\mathrm{Gal}(\mathrm{F}_p^1(N_i)|N_i)
\simeq\mathrm{Gal}(\mathrm{F}_p^2(K)|N_i)/\mathrm{Gal}(\mathrm{F}_p^2(K)|\mathrm{F}_p^1(N_i))
=M_i/\gamma_2(M_i)\]
to the commutator factor group of that maximal normal subgroup \(M_i\)
of the second \(p\)-class group \(G\) of \(K\),
which is associated with the extension \(N_i\)
by the relation \(M_i=\mathrm{Gal}(\mathrm{F}_p^2(K)|N_i)\).
Therefore the \(p\)-class number of \(N_i\),
\(\mathrm{h}_p(N_i)=\lvert M_i/\gamma_2(M_i)\rvert\),
is equal to the order of the commutator factor group of the
corresponding maximal normal subgroup \(M_i\),
which has been determined in theorem \ref{t:MaxCfg}.
According to \cite[p.82]{Bl},
the possible maximum of the invariant \(k\) is dependent on \(m\) and \(p\).
\end{proof}


\subsection{Metabelian \(3\)-groups \(G\) of coclass \(\mathrm{cc}(G)\ge 2\)}
\label{ss:Low}

As in the preceding section, we begin with a purely group theoretic statement
concerning the order of the commutator factor groups \(M_i/\gamma_2(M_i)\)
of the maximal normal subgroups \(M_i\)
of a metabelian \(3\)-group \(G\) of non-maximal class.

\begin{theorem}
\label{t:LowCfg}

Let \(G\) be a metabelian \(3\)-group of coclass \(\mathrm{cc}(G)\ge 2\)
with order \(\lvert G\rvert=3^n\), class \(\mathrm{cl}(G)=m-1\), 
and invariant \(e=n-m+2\ge 3\), where \(4\le m<n\le 2m-3\).
Suppose that the commutator factor group \(G/\gamma_2(G)\) is of type \((3,3)\).
Let generators of \(G=\langle x,y\rangle\) be selected such that
\(\gamma_3(G)=\langle x^3,y^3,\gamma_4(G)\rangle\),
\(x\in G\setminus\chi_s(G)\), if \(s<m-1\),
and \(y\in\chi_s(G)\setminus\gamma_2(G)\).
Assume that the order of the maximal normal subgroups
\(M_i=\langle g_i,\gamma_2(G)\rangle\)
is defined by \(g_1=y\), \(g_2=x\), \(g_3=xy\), and \(g_4=xy^{-1}\).
Finally, let the invariant \(k\) of \(G\) be declared by
\(\lbrack\chi_s(G),\gamma_e(G)\rbrack=\gamma_{m-k}(G)\),
where \(k=0\) for \(m=4\) and \(0\le k\le 1\) for \(m\ge 5\) \cite{Ne}.

Then the order of the commutator factor groups of \(M_1,\ldots,M_4\) is given by

\begin{enumerate}
\item
\(\lvert M_1/\gamma_2(M_1)\rvert=3^{m-k-1}\),
\item
\(\lvert M_2/\gamma_2(M_2)\rvert=3^e\),
\item
\(\lvert M_i/\gamma_2(M_i)\rvert=3^3\) for \(3\le i\le 4\).
\end{enumerate}

\end{theorem}

\begin{proof}

For a \(3\)-group \(G\) of coclass \(\mathrm{cc}(G)\ge 2\)
we also define the main commutator
\(s_2=t_2=\lbrack y,x\rbrack\in\gamma_2(G)\)
as the generator of the cyclic factor \(\gamma_2(G)/\gamma_3(G)\).
But since the factors \(\gamma_j(G)/\gamma_{j+1}(G)\)
are bicyclic  for \(3\le j\le e\), we need two sequences of higher commutators,
which are defined recursively by
\(s_j=\lbrack s_{j-1},x\rbrack=s_{j-1}^{x-1}\in\gamma_j(G)\) and
\(t_j=\lbrack t_{j-1},y\rbrack=t_{j-1}^{y-1}\in\gamma_j(G)\) for \(j\ge 3\).
The structure of the lower central series appears more clearly,
when we start with the third powers \(\sigma_3=y^3\) and \(\tau_3=x^3\),
which generate the first bicyclic factor \(\gamma_3(G)/\gamma_4(G)\),
whereas \((xy)^3\) and \((xy^{-1})^3\) are located in
the second center \(\zeta_2(G)\),
and we recursively form the commutators
\(\sigma_j=\lbrack\sigma_{j-1},x\rbrack\) for \(4\le j\le m\) and
\(\tau_\ell=\lbrack\tau_{\ell-1},y\rbrack\) for \(4\le\ell\le e+2\).
Then, due to the relations
\(\sigma_j^{y-1}=1\) and \(\tau_j^{x-1}=1\) for \(j\ge 3\),
the \(3\)-rank of the non-cyclic factors is bounded by two,
according to \cite[p.47, Th.1.1]{Bl}.
The third powers satisfy the relations
\(\sigma_j^3\sigma_{j+1}^3\sigma_{j+2}=1\) and
\(\tau_j^3\tau_{j+1}^3\tau_{j+2}=1\) for \(j\ge 3\).
The nilpotence of \(G\) is expressed by the relations
\(\sigma_m=1\) and \(\tau_{e+2}=1\).
There is only a single relation between the
sequences \((\sigma_j)_{j\ge 3}\) and \((\tau_j)_{j\ge 3}\)
of the shape \(\tau_{e+1}=\sigma_{m-1}^{-\rho}\) with \(-1\le\rho\le 1\)
in the center \(\zeta_1(G)\).
In the range \(3\le j\le e\) of bicyclic factors,
generators of the members of the lower central series are given by
\(\gamma_j(G)=\langle\sigma_j,\ldots,\sigma_{m-1},\tau_j,\ldots,\tau_e\rangle\),
and in the range \(e+1\le j\le m-1\) of cyclic factors we have
\(\gamma_j(G)=\langle\sigma_j,\ldots,\sigma_{m-1}\rangle\) \cite{Ne}.

Since the commutator group \(\gamma_2(G)\) is contained in each maximal subgroup
\(M_i=\langle g_i,\gamma_2(G)\rangle\)
as a normal subgroup of index \(3\),
the commutator group of \(M_i\) is given by
\(\gamma_2(M_i)=\lbrack M_i,M_i\rbrack=\lbrack\gamma_2(G),M_i\rbrack
=\gamma_2(G)^{g_i-1}
=\langle s_2,\sigma_3,\ldots,\sigma_{m-1},\tau_3,\ldots,\tau_e\rangle^{g_i-1}\),
according to \cite[p.52, Lem.2.1]{Bl}.
Now we are going to use this representation successively for \(1\le i\le 4\),
defining two series of subgroups
\(\Sigma_j=\langle\sigma_\ell\mid\ell\ge j\rangle\)
and
\(T_j=\langle\tau_\ell\mid\ell\ge j\rangle\)
for \(j\ge 3\).

\begin{enumerate}

\item
With \(Y=y-1\), we have
\(\gamma_2(M_1)=\gamma_2(G)^{y-1}
=\langle s_2,\sigma_3,\ldots,\sigma_{m-1},\tau_3,\ldots,\tau_e\rangle^Y\)\\
\(=\langle s_2^Y,\sigma_3^Y,\ldots,\sigma_{m-1}^Y,\tau_3^Y,\ldots,\tau_e^Y\rangle
=\langle t_3,1,\ldots,1,\tau_4,\ldots,\tau_{e+1}\rangle
=\langle t_3,T_4\rangle\).

\item
With \(X=x-1\), we have
\(\gamma_2(M_2)=\gamma_2(G)^{x-1}
=\langle s_2,\sigma_3,\ldots,\sigma_{m-1},\tau_3,\ldots,\tau_e\rangle^X\)\\
\(=\langle s_2^X,\sigma_3^X,\ldots,\sigma_{m-2}^X,\sigma_{m-1}^X,\tau_3^X,\ldots,\tau_e^X\rangle
=\langle s_3,\sigma_4,\ldots,\sigma_{m-1},1,1,\ldots,1\rangle
=\langle s_3,\Sigma_4\rangle\).

\item
For
\(\gamma_2(M_3)=\gamma_2(G)^{xy-1}
=\langle s_2,\sigma_3,\ldots,\sigma_{m-1},\tau_3,\ldots,\tau_e\rangle^{xy-1}\),\\
we obtain, according to the right product rule for commutators,\\
\(\sigma_j^{xy-1}=\lbrack\sigma_j,xy\rbrack
=\lbrack\sigma_j,y\rbrack\cdot\lbrack\sigma_j,x\rbrack\cdot\lbrack\lbrack\sigma_j,x\rbrack,y\rbrack
=1\cdot\sigma_{j+1}\cdot\lbrack\sigma_{j+1},y\rbrack=\sigma_{j+1}\)
for \(3\le j\le m-1\),\\
\(\tau_\ell^{xy-1}=\lbrack\tau_\ell,xy\rbrack
=\lbrack\tau_\ell,y\rbrack\cdot\lbrack\tau_\ell,x\rbrack\cdot\lbrack\lbrack\tau_\ell,x\rbrack,y\rbrack
=\tau_{\ell+1}\cdot 1\cdot\lbrack 1,y\rbrack=\tau_{\ell+1}\)
for \(3\le\ell\le e+1\),\\
and
\(s_2^{xy-1}=\lbrack s_2,xy\rbrack
=\lbrack s_2,y\rbrack\cdot\lbrack s_2,x\rbrack\cdot\lbrack\lbrack s_2,x\rbrack,y\rbrack
=t_3s_3s_2^{XY}\equiv t_3s_3\pmod{\gamma_4(G)}\),\\
since \(s_2^X=s_3\in\gamma_3(G)\) and thus \(s_2^{XY}=\lbrack s_3,y\rbrack\in\gamma_4(G)\).\\
Therefore
\(\gamma_2(M_3)
=\langle s_3t_3s_2^{XY},\sigma_4,\ldots,\sigma_{m-1},\tau_4,\ldots,\tau_{e+1}\rangle
=\langle s_3t_3s_2^{XY},\Sigma_4T_4\rangle
=\langle s_3t_3,\gamma_4(G)\rangle\).

\item
For
\(\gamma_2(M_4)=\gamma_2(G)^{xy^{-1}-1}
=\langle s_2,\sigma_3,\ldots,\sigma_{m-1},\tau_3,\ldots,\tau_e\rangle^{xy^{-1}-1}\),\\
we use the right product rule and the rule for the inverse to obtain\\
\(\sigma_j^{xy^{-1}-1}=\lbrack\sigma_j,xy^{-1}\rbrack
=\lbrack\sigma_j,y^{-1}\rbrack\cdot\lbrack\sigma_j,x\rbrack\cdot\lbrack\lbrack\sigma_j,x\rbrack,y^{-1}\rbrack
=\lbrack\sigma_j,y^{-1}\rbrack\cdot\sigma_{j+1}\cdot\lbrack\sigma_{j+1},y^{-1}\rbrack\)\\
\(=\lbrack\sigma_j,y\rbrack^{-y^{-1}}\cdot\sigma_{j+1}\cdot\lbrack\sigma_{j+1},y\rbrack^{-y^{-1}}
=1\cdot\sigma_{j+1}\cdot 1=\sigma_{j+1}\)
for \(3\le j\le m-1\),\\
\(\tau_\ell^{xy^{-1}-1}
=\lbrack\tau_\ell,xy^{-1}\rbrack
=\lbrack\tau_\ell,y^{-1}\rbrack\cdot\lbrack\tau_\ell,x\rbrack\cdot\lbrack\lbrack\tau_\ell,x\rbrack,y^{-1}\rbrack
=\lbrack\tau_\ell,y^{-1}\rbrack\cdot 1\cdot\lbrack 1,y^{-1}\rbrack
=\lbrack\tau_\ell,y^{-1}\rbrack\)\\
for \(3\le\ell\le e+1\), and\\
\(s_2^{xy^{-1}-1}=\lbrack s_2,xy^{-1}\rbrack
=\lbrack s_2,y^{-1}\rbrack\cdot\lbrack s_2,x\rbrack\cdot\lbrack\lbrack s_2,x\rbrack,y^{-1}\rbrack
=\lbrack s_2,y^{-1}\rbrack\cdot s_3\cdot\lbrack s_3,y^{-1}\rbrack
\equiv\lbrack s_2,y^{-1}\rbrack s_3\)\\
\(=s_3\lbrack s_2,y\rbrack^{-y^{-1}}
=s_3\left(t_3^{-1}\right)^{y^{-1}}=s_3t_3^{-1}t_3yt_3^{-1}y^{-1}
=s_3t_3^{-1}\lbrack t_3^{-1},y^{-1}\rbrack
\equiv s_3t_3^{-1}\pmod{\gamma_4(G)}\),\\
because
\(s_3\in\gamma_3(G)\) and \(t_3^{-1}\in\gamma_3(G)\)
imply that
\(\lbrack s_3,y^{-1}\rbrack\in\gamma_4(G)\)
and \(\lbrack t_3^{-1},y^{-1}\rbrack\in\gamma_4(G)\).\\
Now we have to investigate the commutators \(\lbrack\tau_\ell,y^{-1}\rbrack\) with \(3\le\ell\le e+1\).
We start with
\(\langle\lbrack\tau_{e+1},y^{-1}\rbrack\rangle
=\langle\lbrack\sigma_{m-1}^{-\rho},y^{-1}\rbrack\rangle=1=T_{e+2}\),
since \(\sigma_{m-1}\in\zeta_1(G)\),
and show by descending induction for \(e\ge\ell\ge 3\) that\\
\(\lbrack\tau_\ell,y^{-1}\rbrack=\lbrack\tau_\ell,y\rbrack^{-y^{-1}}
=\tau_{\ell+1}^{-y^{-1}}=\tau_{\ell+1}^{-1}\tau_{\ell+1}y\tau_{\ell+1}^{-1}y^{-1}
=\tau_{\ell+1}^{-1}\lbrack\tau_{\ell+1}^{-1},y^{-1}\rbrack\)\\
\(\in\langle\tau_{\ell+1}^{-1},T_{\ell+2}\rangle
=\langle\tau_{\ell+1},T_{\ell+2}\rangle=T_{\ell+1}\),
and in mutual dependence\\
\(\lbrack\tau_\ell^{-1},y^{-1}\rbrack
=\lbrack y,\tau_\ell^{-1}\rbrack^{y^{-1}}
=\left(\lbrack\tau_\ell,y\rbrack^{\tau_\ell^{-1}}\right)^{y^{-1}}
=\left(\tau_{\ell+1}^{\tau_\ell^{-1}}\right)^{y^{-1}}
=\tau_{\ell+1}^{y^{-1}}=\tau_{\ell+1}\tau_{\ell+1}^{-1}y\tau_{\ell+1}y^{-1}
=\tau_{\ell+1}\lbrack\tau_{\ell+1},y^{-1}\rbrack
\in\langle\tau_{\ell+1},T_{\ell+2}\rangle=T_{\ell+1}\),
and thus\\
\(\langle\lbrack\tau_\ell,y^{-1}\rbrack,\lbrack\tau_{\ell+1},y^{-1}\rbrack,\ldots,\lbrack\tau_e,y^{-1}\rbrack\rangle
=\langle\lbrack\tau_\ell,y^{-1}\rbrack,T_{\ell+2}\rangle
=\langle\tau_{\ell+1}^{-1}\lbrack\tau_{\ell+1}^{-1},y^{-1}\rbrack,T_{\ell+2}\rangle\)\\
\(=\langle\tau_{\ell+1}^{-1},T_{\ell+2}\rangle
=\langle\tau_{\ell+1},T_{\ell+2}\rangle=T_{\ell+1}\).\\
Therefore
\(\gamma_2(M_4)
=\langle s_3t_3^{-1},\sigma_4,\ldots,\sigma_{m-1},\lbrack\tau_3,y^{-1}\rbrack,\ldots,\lbrack\tau_e,y^{-1}\rbrack\rangle
=\langle s_3t_3^{-1},\Sigma_4T_4\rangle
=\langle s_3t_3^{-1},\gamma_4(G)\rangle\).

\end{enumerate}

To determine the order of the groups \(\gamma_2(M_i)\),
we first show generally that
for a sequence of group elements \((\upsilon_\ell)_{\ell\ge a}\),
satisfying the relations
\(\upsilon_\ell^3\upsilon_{\ell+1}^3\upsilon_{\ell+2}=1\) for \(\ell\ge a\)
and
\(\upsilon_\ell=1\) for \(\ell\ge b\)
with integers
\(b\ge a\),
the order of the subgroups
\(Y_j=\langle\upsilon_\ell\mid\ell\ge j\rangle\)
is bounded by
\(\lvert Y_j\rvert\le 3^{b-j}\) for \(a\le j\le b\).
We start with
\(\lvert Y_b\rvert=\lvert 1\rvert=1\le 3^0\)
and obtain by descending induction for \(b>j\ge a\) that
\(\lvert Y_j\rvert=\lvert\langle\upsilon_j,Y_{j+1}\rangle\rvert
\le 3\lvert Y_{j+1}\rvert\le 3\cdot 3^{b-(j+1)}=3^{b-j}\),
since \(\upsilon_j^3=\upsilon_{j+2}^{-1}\upsilon_{j+1}^{-3}\in Y_{j+1}\).\\
Thus we get three sequences of estimates, where always \(a=3\):

\begin{itemize}
\item
putting \(\upsilon_\ell=\sigma_\ell\), \(b=m\) we have
\(\lvert\Sigma_j\rvert\le 3^{m-j}\) for \(3\le j\le m\),
\item
putting \(k=0\) and \(\upsilon_\ell=\tau_\ell\), \(b=e+1\) we have
\(\lvert T_j\rvert\le 3^{e+1-j}\) for \(3\le j\le e+1\),
\item
putting \(k=1\) and \(\upsilon_\ell=\tau_\ell\), \(b=e+2\) we have
\(\lvert T_j\rvert\le 3^{e+2-j}\) for \(3\le j\le e+2\).
\end{itemize}

In the case \(k=1\), the single relation between the sequences
\((\sigma_\ell)_{\ell\ge 3}\) and \((\tau_\ell)_{\ell\ge 3}\)
is \(\tau_{e+1}=\sigma_{m-1}^{-\rho}\)
in the cyclic center \(\zeta_1(G)=\gamma_{m-1}(G)\).
However, in the case \(k=0\) with bicyclic center \(\zeta_1(G)\ge\gamma_{m-1}(G)\)
there is no relation.
Together we have \(\Sigma_j\cap T_j=\gamma_{m-k}(G)\) for \(3\le j\le e+1\).

Calculating the order of the members \(\gamma_i(G)\) with \(3\le i\le m\)
of the lower central series,
we can prove now that the estimates above are actually sharp.

In the case \(k=0\), we have
\(\Sigma_j\cap T_j=1\) for all \(3\le j\le m-1\).\\
In the range \(e+1\le j\le m\) of cyclic factors, we have
\(\gamma_j(G)=\langle\sigma_j,\ldots,\sigma_{m-1}\rangle=\Sigma_j\),\\
thus on the one hand
\(\lvert\gamma_j(G)\rvert=\lvert\Sigma_j\rvert\le 3^{m-j}\),\\
on the other hand
\(\lvert\gamma_j(G)\rvert=\prod_{\ell=j}^{m-1}(\gamma_\ell(G):\gamma_{\ell+1}(G))
=3^{(m-1)-(j-1)}=3^{m-j}\),\\
and consequently \(\lvert\Sigma_j\rvert=3^{m-j}\).\\
In the range \(3\le j\le e\) of bicyclic factors, we have
\(\gamma_j(G)=\langle\sigma_j,\ldots,\sigma_{m-1},\tau_j,\ldots,\tau_e\rangle
=\Sigma_j\times T_j\),\\
thus on the one hand
\(\lvert\gamma_j(G)\rvert=\lvert\Sigma_j\rvert\lvert T_j\rvert
\le 3^{m-j}\cdot 3^{e+1-j}=3^{m+e-2j+1}\),\\
on the other hand
\(\lvert\gamma_j(G)\rvert
=\lvert\gamma_{e+1}(G)\rvert\prod_{\ell=j}^{e}(\gamma_\ell(G):\gamma_{\ell+1}(G))
=3^{m-(e+1)}\cdot(3^2)^{e-(j-1)}=3^{m+e-2j+1}\),\\
and together \(\lvert\Sigma_j\rvert=3^{m-j}\) and \(\lvert T_j\rvert=3^{e+1-j}\).

In the case \(k=1\), we have \(e\le m-2\) and
\(\Sigma_j\cap T_j=\gamma_{m-1}(G)=\langle\sigma_{m-1}\rangle\) for \(3\le j\le e+1\)
with \(\sigma_{m-1}^3=1\).\\
In the range \(e+2\le j\le m\) of cyclic factors,
we obtain, similarly as for \(k=0\), that\\
\(3^{m-j}=\lvert\gamma_j(G)\rvert=\lvert\Sigma_j\rvert\le 3^{m-j}\)
and thus \(\lvert\Sigma_j\rvert=3^{m-j}\).\\
However, in the range \(3\le j\le e\) of bicyclic factors
and for \(j=e+1\),
the product \(\gamma_j(G)=\Sigma_j T_j\) is not direct
and we have\\
\(3^{m+e-2j+1}=\lvert\gamma_j(G)\rvert
=\lvert\Sigma_j\rvert\lvert T_j\rvert/\lvert\Sigma_j\cap T_j\rvert
\le 3^{m-j}\cdot 3^{e+2-j}\cdot 3^{-1}=3^{m+e-2j+1}\)\\
and thus \(\lvert\Sigma_j\rvert=3^{m-j}\) and \(\lvert T_j\rvert=3^{e+2-j}\).

Finally we obtain the order of the commutator factor groups
\(M_i/\gamma_2(M_i)\)
by means of the following consideration.\\
For an arbitrary element \(z\in G\setminus\gamma_2(G)\),
the commutator \(\lbrack s_2,z\rbrack\) cannot lie in \(\gamma_4(G)\),
since otherwise \(z\in\chi_2(G)\setminus\gamma_2(G)\) and thus \(s=2\),
which is only possible for a group \(G\) of maximal class.
Consequently
\(s_3=\lbrack s_2,x\rbrack\not\in\Sigma_4\le\gamma_4(G)\)
and
\(t_3=\lbrack s_2,y\rbrack\not\in T_4\le\gamma_4(G)\).
The elements
\(s_3t_3,s_3t_3^{-1}\)
cannot belong to \(\gamma_4(G)\) either,
since we have seen above that
\(\lbrack s_2,xy\rbrack\equiv s_3t_3\pmod{\gamma_4(G)}\) and
\(\lbrack s_2,xy^{-1}\rbrack\equiv s_3t_3^{-1}\pmod{\gamma_4(G)}\).\\
On the other hand, the third powers of these elements satisfy
\(s_3^3=\sigma_5\in\Sigma_4\) and 
\(t_3^3=\tau_5^{-1}\in T_4\),
due to the connecting relations 
\(s_j^3=\sigma_{j+2}\) and 
\(t_j^3=\tau_{j+2}^{-1}\)
for \(j\ge 3\)
\cite{Ne}.
Therefore also
\(s_3^3t_3^3,s_3^3t_3^{-3}\in\Sigma_4T_4=\gamma_4(G)\).\\
Finally, using \(n-e+2=m\) and \(n-m+2=e\), we obtain\\
\(\lvert M_1/\gamma_2(M_1)\rvert=\lvert M_1\rvert/3\lvert T_4\rvert=3^{n-1}/3^{1+e+1+k-4}=3^{m-k-1}\),\\
\(\lvert M_2/\gamma_2(M_2)\rvert=\lvert M_2\rvert/3\lvert\Sigma_4\rvert=3^{n-1}/3^{1+m-4}=3^e\),\\
and
\(\lvert M_i/\gamma_2(M_i)\rvert=\lvert M_i\rvert/3\lvert\gamma_4(G)\rvert=3^{n-1}/3^{1+n-5}=3^3\)
for \(3\le i\le 4\),\\
since
\(3^n=\lvert G\rvert
=(G:\gamma_2(G))(\gamma_2(G):\gamma_3(G))(\gamma_3(G):\gamma_4(G))\lvert\gamma_4(G)\rvert
=3^2\cdot 3\cdot 3^2\cdot\lvert\gamma_4(G)\rvert\).
\end{proof}

\begin{corollary}
\label{c:LowCg}

Under the assumptions of theorem \ref{t:LowCfg},
the commutator groups
of the maximal normal subgroups \(M_1,\ldots,M_4\)
of a metabelian \(3\)-group \(G=\langle x,y\rangle\) of coclass \(\mathrm{cc}(G)\ge 2\)
with \(G/\gamma_2(G)\) of type \((3,3)\),
\(\lvert G\rvert=3^n\), \(\mathrm{cl}(G)=m-1\), \(e=n-m+2\ge 3\),
\(4\le m<n\le 2m-3\), and 
\(\lbrack\chi_s(G),\gamma_e(G)\rbrack=\gamma_{m-k}(G)\), \(0\le k\le 1\)
are given by

\[\begin{array}{rclcrcl}
\gamma_2(M_1)&=&\langle t_3,\tau_4,\ldots,\tau_{e+1}\rangle&\text{ with }&\lvert\gamma_2(M_1)\rvert&=&3^{e+k-2},\\
\gamma_2(M_2)&=&\langle s_3,\sigma_4,\ldots,\sigma_{m-1}\rangle&\text{ with }&\lvert\gamma_2(M_2)\rvert&=&3^{m-3},\\
\gamma_2(M_3)&=&\langle s_3t_3,\gamma_4(G)\rangle&\text{ with }&\lvert\gamma_2(M_3)\rvert&=&3^{n-4},\\
\gamma_2(M_4)&=&\langle s_3t_3^{-1},\gamma_4(G)\rangle&\text{ with }&\lvert\gamma_2(M_4)\rvert&=&3^{n-4},
\end{array}\]

\noindent
where
\(s_3=\lbrack s_2,x\rbrack\), \(t_3=\lbrack s_2,y\rbrack\),
\(s_3t_3\), and \(s_3t_3^{-1}\)
lie in \(\gamma_3(G)\setminus\gamma_4(G)\),
if \(s_2=\lbrack y,x\rbrack\in\gamma_2(G)\) denotes the main commutator, and
\(\gamma_4(G)=\langle\sigma_4,\ldots,\sigma_{m-1},\tau_4,\ldots,\tau_{e+1}\rangle\)
with generators, which are defined recursively by
\(\sigma_j=\lbrack\sigma_{j-1},x\rbrack\) for \(4\le j\le m-1\) and
\(\tau_\ell=\lbrack\tau_{\ell-1},y\rbrack\) for \(4\le\ell\le e+1\),
starting with \(\sigma_3=y^3\) and \(\tau_3=x^3\).

\noindent
None of the maximal normal subgroups \(M_1,\ldots,M_4\) is abelian.

\end{corollary}

Now we come to the number theoretic application
of theorem \ref{t:LowCfg},
reducing the determination of the \(3\)-class numbers
\(\mathrm{h}_3(N_i)\)
to the purely group theoretic preliminaries in this section,
with the aid of Artin's reciprocity law \cite{Ar1}.

\begin{theorem}
\label{t:LowDir}

Let \(K\) be an arbitrary base field
with \(3\)-class group \(\mathrm{Cl}_3(K)\) of type \((3,3)\).
Suppose that the second \(3\)-class group \(G=\mathrm{Gal}(\mathrm{F}_3^2(K)\vert K)\)
is of coclass \(\mathrm{cc}(G)\ge 2\)
with order \(\lvert G\rvert=3^n\),
class \(\mathrm{cl}(G)=m-1\),
and invariant \(e=n-m+2\ge 3\),
where \(4\le m<n\le 2m-3\).
Under the assumptions of theorem \ref{t:LowCfg}
for the generators of \(G\),
the \(3\)-class numbers of the
four unramified cyclic cubic extension fields \(N_1,\ldots,N_4\) of \(K\)
are given by

\begin{eqnarray*}
\mathrm{h}_3(N_1)&=&
\begin{cases}
3^{m-1},&\text{ if }\lbrack\chi_s(G),\gamma_e(G)\rbrack=1,\ k=0,\ m\ge 4,\\
3^{m-2},&\text{ if }\lbrack\chi_s(G),\gamma_e(G)\rbrack=\gamma_{m-1}(G),\ k=1,\ m\ge 5,
\end{cases}\\
\mathrm{h}_3(N_2)&=&3^e,\\
\mathrm{h}_3(N_i)&=&3^3\text{ for }3\le i\le 4.
\end{eqnarray*}

\end{theorem}

\begin{proof}

According to the Artin isomorphism \cite[p.361]{Ar1}
\[\mathrm{Cl}_3(N_i)\simeq\mathrm{Gal}(\mathrm{F}_3^1(N_i)|N_i),\]
the \(3\)-class group of the extension field \(N_i\) is isomorphic to the
relative Galois group of the Hilbert \(3\)-class field of \(N_i\) over \(N_i\)
for \(1\le i\le 4\) (see also Miyake \cite[p.297]{My}).
By Galois theory, we have a further isomorphism
\[\mathrm{Gal}(\mathrm{F}_3^1(N_i)|N_i)
\simeq\mathrm{Gal}(\mathrm{F}_3^2(K)|N_i)/\mathrm{Gal}(\mathrm{F}_3^2(K)|\mathrm{F}_3^1(N_i))
=M_i/\gamma_2(M_i)\]
to the commutator factor group of the uniquely determined maximal normal subgroup \(M_i\)
of the second \(3\)-class group \(G\) of \(K\),
which is associated with the extension \(N_i\)
by the relation \(M_i=\mathrm{Gal}(\mathrm{F}_3^2(K)|N_i)\).
Consequently the \(3\)-class number of \(N_i\),
\(\mathrm{h}_3(N_i)=\lvert M_i/\gamma_2(M_i)\rvert\),
coincides with the order of the commutator factor group of the
corresponding maximal normal subgroup \(M_i\), which
has been determined in theorem \ref{t:LowCfg}.
By \cite{Ne},
the possible maximum of the invariant \(k\) is dependent on \(m\).
\end{proof}


\section{Restrictions for class and coclass, enforced by quadratic base fields}
\label{s:Qdr}

All of our previous results generally
concern an arbitrary base field \(K\).
In this section let
\(K=\mathbb{Q}(\sqrt{D})\) be a quadratic number field,
\(p\ge 3\) an odd prime,
and \(N\vert K\) an unramified cyclic extension of relative degree \(p\).

\begin{remark}
\label{r:Hil94}
According to Hilbert's theorem \(94\)
\cite[p.279]{Hi},
the existence of an unramified cyclic extension field \(N\vert K\)
of prime degree \(p\) implies the
divisibility of the class number \(\mathrm{h}(K)\) of the base field \(K\) by \(p\).
\end{remark}


\subsection{Dihedral extensions}
\label{ss:QdrDih}

In the present situation,
we apply the theory of the absolute extension \(N\vert\mathbb{Q}\)
to obtain more sophisticated statements about \(p\)-class numbers.

\begin{proposition}
\label{p:QdrDih}

Let \(K\) be a quadratic number field and \(p\) an odd prime.
For an unramified cyclic extension \(N\vert K\) of relative degree \(p\),
\(N\vert\mathbb{Q}\) is a non-abelian absolute Galois extension
with automorphism group \(\mathrm{Gal}(N\vert\mathbb{Q})\)
isomorphic to the dihedral group \(D(2p)\) of order \(2p\).

\end{proposition}

\begin{proof}

The conductor
of the unramified cyclic extension \(N\vert K\)
is \(f=1\).
Therefore the \(p\)-ray class group modulo \(f\) of \(K\),
whose subgroups of index \(p\)
are in one to one correspondence with the cyclic extensions \(N\vert K\)
of relative degree \(p\) with conductor dividing \(f\)
by \cite[p.361, Allgemeines Reziprozit\"atsgesetz]{Ar1} and \cite[p.836]{Ma2},
coincides with the ordinary \(p\)-class group \(\mathrm{Cl}_p(K)\) of \(K\).

We denote by \(\tau\) the generating automorphism
of \(\mathrm{Gal}(K\vert\mathbb{Q})\).
Since \(K\) is quadratic, \(\tau\) is of order \(2\),
and since the norm map \(\mathrm{Norm}_{K\vert\mathbb{Q}}\)
of \(\mathrm{Cl}_p(K)\) has
the trivial group \(\mathrm{Cl}(\mathbb{Q})=1\) as its image,
every class \(c\in\mathrm{Cl}_p(K)\) satisfies the relation
\(c\cdot c^\tau=c^{1+\tau}\hat{=}\mathrm{Norm}_{K\vert\mathbb{Q}}(c)=1\)
and thus \(c^\tau=c^{-1}\).

The norm class group \(H=\mathrm{Norm}_{N\vert K}(\mathrm{Cl}_p(N))\)
of index \(p\) in \(\mathrm{Cl}_p(K)\),
associated with \(N\) by the Artin-Galois correspondence,
is invariant under \(\tau\),
since \(c\in H\) implies \(c^\tau=c^{-1}\in H\).
However, the cosets of \(H\) cannot remain invariant under \(\tau\).
Otherwise, for a coset representative \(c\) with
\(\mathrm{Cl}_p(K)=H\cup cH\cup c^2H\cup\ldots\cup c^{p-1}H\),
that is, of order \(p\ge 3\) with respect to \(H\),
the invariance \(cH=c^\tau H=c^{-1}H\) would cause
the contradiction \(c^2H=cH\cdot c^{-1}H=H\).
Therefore, according to the lemma \cite[p.572]{Ha} of Hasse,
the extension \(N\vert\mathbb{Q}\) is Galois but not abelian.

Consequently, the unique possibility for the non-abelian group
\(\mathrm{Gal}(N\vert\mathbb{Q})\) of order
\(\lbrack N:\mathbb{Q}\rbrack=\lbrack N:K\rbrack\cdot\lbrack K:\mathbb{Q}\rbrack=2p\)
is the dihedral group \(D(2p)\).
\end{proof}

Now we assume that \(\mathrm{Gal}(N\vert\mathbb{Q})\)
is generated by automorphisms \(\sigma,\tau\)
with the relations \(\sigma^p=1\), \(\tau^2=1\), \(\sigma\tau=\tau\sigma^{-1}\),
and we denote by \(L\) the non-Galois subfield of \(N\),
which is fixed by the subgroup \(\langle\tau\rangle\).
\(L\) is of absolute degree \(p\) over \(\mathbb{Q}\).


\subsection{Galois cohomology of unit groups}
\label{ss:QdrCoh}

For an algebraic number field \(F\) let
\(\mathcal{I}(F)\) denote the group of fractional ideals,
\(\mathcal{P}(F)\) the subgroup of principal ideals,
and \(U(F)\) the group of units
of the maximal order \(\mathcal{O}(F)\).

For an extension field \(X\vert F\) let
\(\mathrm{Norm}_{X\vert F}\) denote the relative norm map
and \(E(X\vert F)=U(X)\cap\mathrm{Ker}(\mathrm{Norm}_{X\vert F})\)
the group of relative units of \(X\vert F\),
that is, the units \(E\in U(X)\) with relative norm \(\mathrm{Norm}_{X\vert F}(E)=1\).

\begin{proposition}
\label{p:QdrCoh}
Let \(N\) be an unramified cyclic extension
of odd prime degree \(p\ge 3\) of a quadratic base field \(K\)
and denote by \(U(N)\) the unit group of \(N\).

\begin{enumerate}

\item
For a complex quadratic field \(K\)
the structure of \(U(N)\) as a Galois module over \(\mathrm{Gal}(N\vert\mathbb{Q})\)
in the sense of Moser \cite{Mo} is of type \(\alpha\).

\item
For a real quadratic field \(K\)
the structure of \(U(N)\) as a Galois module over \(\mathrm{Gal}(N\vert\mathbb{Q})\)
in the sense of Moser is\\
of type \(\delta\), if the cohomology \(\mathrm{H}^0(\mathrm{Gal}(N\vert K),U(N))\) is trivial, and\\
of type \(\alpha\), if \(\lvert\mathrm{H}^0(\mathrm{Gal}(N\vert K),U(N))\rvert=p\).

\end{enumerate}

\end{proposition}

\begin{definition}
\label{d:QdrCoh}
We shortly say that \(L\) or \(N\) is \textit{of type} \(\alpha\) respectively \(\delta\),
if the structure of \(U(N)\) as a Galois module over \(\mathrm{Gal}(N\vert\mathbb{Q})\)
is of that type.
\end{definition}

\begin{remark}
\label{r:NrmRes}
Since the \(0\)th cohomology group \(\mathrm{H}^0(\mathrm{Gal}(N\vert K),U(N))\)
coincides with the factor group \(U(K)/\mathrm{Norm}_{N\vert K}U(N)\),
the quadratic fundamental unit \(\eta\in U(K)\)
is the norm \(\eta=\mathrm{Norm}_{N\vert K}(H)\) of a unit \(H\in U(N)\),
if \(N\) is of type \(\delta\).
There is, however,
only the trivial norm relation \(\eta^p=\mathrm{Norm}_{N\vert K}(\eta)\),
if \(N\) is of type \(\alpha\).
\end{remark}

\begin{proof}

We denote by \(\mathcal{G}\) the relative Galois group
\(\mathrm{Gal}(N\vert K)=\langle\sigma\rangle\)
and by \(t\) the number of prime ideals of \(K\) which ramify in \(N\).
According to Hilbert's theorem \(93\) \cite[p.277]{Hi},
the group \(\mathcal{I}(N)^\mathcal{G}/\mathcal{I}(K)\)
of primitive ambiguous ideals of \(N\)
is \(p\)-elementary abelian of type \((\overbrace{p,\ldots,p}^{t\text{ times}})\).
For an unramified extension \(N\vert K\)
we have \(t=0\) and thus \(\mathcal{I}(N)^\mathcal{G}=\mathcal{I}(K)\).

In the ramified case \(t\ge 1\),
the units of \(N\) cannot be generated completely by the units
of the subfield \(L\) and of its conjugates \(L^\sigma,\ldots,L^{\sigma^{p-1}}\).
According to Schmithals \cite[p.57, Satz 8]{Sm},
the index \((E(N\vert K):E_0)\) of the product group
\(E_0=\prod_{i=0}^{p-1}\,E(L^{\sigma^i}\vert\mathbb{Q})\)
of norm-positive units of all non-Galois subfields of \(N\)
in the group of relative units \(E(N\vert K)\)
is equal to the number \(\lvert\mathcal{P}(L)^\mathcal{G}/\mathcal{P}(\mathbb{Q})\rvert\)
of primitive ambiguous principal ideals of \(L\),
provided that the quadratic discriminant \(d(K)\) differs from \(-3\)
in the case \(p=3\).

However, for an unramified extension \(N\vert K\)
we have \(\mathcal{I}(N)^\mathcal{G}=\mathcal{I}(K)\),\\
and thus
\(\mathcal{P}(N)^\mathcal{G}
=\mathcal{P}(N)\cap\mathcal{I}(N)^\mathcal{G}=\mathcal{P}(N)\cap\mathcal{I}(K)\)\\
and
\(\mathcal{P}(L)^\mathcal{G}=\mathcal{P}(L)\cap\mathcal{P}(N)^\mathcal{G}
=\mathcal{P}(L)\cap\mathcal{P}(N)\cap\mathcal{I}(K)
=\mathcal{P}(L)\cap\mathcal{I}(K)=\mathcal{P}(\mathbb{Q})\)\\
and therefore \((E(N\vert K):E_0)
=\lvert\mathcal{P}(L)^\mathcal{G}/\mathcal{P}(\mathbb{Q})\rvert=1\).

\begin{enumerate}

\item
For a complex quadratic base field \(K\),
the structure of \(U(N)\) as a Galois module over \(\mathrm{Gal}(N\vert\mathbb{Q})\)
is determined uniquely by the index \((E(N\vert K):E_0)\),
according to \cite[p.61, Prop.III.3]{Mo} or \cite[p.55, Satz 5]{Sm}.
Since it just turned out that \((E(N\vert K):E_0)=1\) in the unramified case,
\(N\) is of type \(\alpha\) in the sense of \cite[p.61, Def.III.1]{Mo}.

\item
To characterise 
the structure of \(U(N)\) as a Galois module over \(\mathrm{Gal}(N\vert\mathbb{Q})\)
uniquely for a real quadratic base field \(K\),
it is not sufficient to know the index \((E(N\vert K):E_0)\)
but we additionally need the norm index of unit groups
\((U(K):\mathrm{Norm}_{N\vert K}U(N))\).
Since \(N\vert K\) is unramified,
we have \((E(N\vert K):E_0)=1\) and
\(N\) is of type \(\delta\), if \((U(K):\mathrm{Norm}_{N\vert K}U(N))=1\),
and of type \(\alpha\), if \((U(K):\mathrm{Norm}_{N\vert K}U(N))=p\),
in the sense of \cite[p.62, Th.III.5]{Mo} or \cite[p.55, Satz 5]{Sm}.
\end{enumerate}
\end{proof}


\subsection{Principalisation of ideal classes}
\label{ss:QdrPrc}

For a number field \(F\) with maximal order \(\mathcal{O}(F)\)
we denote by
\(\mathrm{Cl}(F)=\mathcal{I}(F)/\mathcal{P}(F)\) the ideal class group of \(F\).

If there exists an ideal \(\mathfrak{a}\in\mathcal{I}(K)\) in a base field \(K\),
whose ideal class \(\mathfrak{a}\mathcal{P}(K)\)
is different from the principal class \(\mathcal{P}(K)\),
but whose extension ideal in an extension field \(N\) of \(K\)
is a principal ideal,
\(\mathfrak{a}\mathcal{O}(N)=A\mathcal{O}(N)\) with a number \(A\in N\),
and thus belongs to the principal class \(\mathcal{P}(N)\),
then we speak about \textit{principalisation} in the field extension \(N\vert K\).
This phenomenon is described most adequately by
the kernel of the class extension homomorphism
\[\mathrm{j}_{N\vert K}:\mathrm{Cl}(K)\longrightarrow\mathrm{Cl}(N),
\ \mathfrak{a}\mathcal{P}(K)\mapsto(\mathfrak{a}\mathcal{O}(N))\mathcal{P}(N),\]
which is induced by the natural extension monomorphism
\(\mathcal{I}(K)\longrightarrow\mathcal{I}(N)\),
\(\mathfrak{a}\mapsto\mathfrak{a}\mathcal{O}(N)\)
of ideals.

In the present situation of an
unramified cyclic extension \(N\) of odd prime degree \(p\ge 3\)
of a quadratic base field \(K\) with arbitrary positive \(p\)-class rank
we can specify the structure of the \textit{principalisation kernel}
\(\mathrm{Ker}(\mathrm{j}_{N\vert K})\) exactly.

\begin{proposition}
\label{p:QdrPrc}
Let \(N\) be an unramified cyclic extension
of odd prime degree \(p\ge 3\)
of a quadratic base field \(K\) and denote by
\(\mathrm{j}_{N\vert K}:\mathrm{Cl}_p(K)\longrightarrow\mathrm{Cl}_p(N)\),
\(\mathfrak{a}\mathcal{P}(K)\mapsto(\mathfrak{a}\mathcal{O}(N))\mathcal{P}(N)\)
the extension homomorphism of \(p\)-classes.

\begin{enumerate}

\item
The principalisation kernel \(\mathrm{Ker}(\mathrm{j}_{N\vert K})\)
is a \(p\)-elementary abelian subgroup of the \(p\)-class group \(\mathrm{Cl}_p(K)\).

\item
For a complex quadratic field \(K\),
\(\mathrm{Ker}(\mathrm{j}_{N\vert K})\) is cyclic of order \(p\).

\item
For a real quadratic field \(K\),
\[\mathrm{Ker}(\mathrm{j}_{N\vert K})\text{ is }
\begin{cases}
\text{ cyclic of order }p,&\text{ if }N\text{ is of type }\delta\text{ (partial principalisation)},\\
\text{ bicyclic of type }(p,p),&\text{ if }N\text{ is of type }\alpha\text{ (total principalisation)}.
\end{cases}
\]

\end{enumerate}

\end{proposition}

\begin{proof}

If the extension \(N\vert K\) is unramified, then
\(\mathrm{Ker}(\mathrm{j}_{N\vert K})
=\mathcal{I}(K)\cap\mathcal{P}(N)/\mathcal{P}(K)\)
coincides with the group \(\mathcal{P}(N)^\mathcal{G}/\mathcal{P}(K)\)
of primitive ambiguous principal ideals of \(N\),
where we denote by \(\mathcal{G}\) the relative Galois group
\(\mathrm{Gal}(N\vert K)=\langle\sigma\rangle\).
Further,
\(\mathcal{P}(N)^\mathcal{G}/\mathcal{P}(K)\)
is isomorphic to the factor group \(E(N\vert K)/U(N)^{1-\sigma}\)
of the relative units of \(N\vert K\)
by the symbolic \((1-\sigma)\)th powers of units of \(N\),
by \cite[p.275, (5),(6)]{Ha1}.
Finally, according to \cite[p.268, Satz 12]{Ha1},
the order of \(E(N\vert K)/U(N)^{1-\sigma}\) is given by
\((E(N\vert K):U(N)^{1-\sigma})=\frac{p\cdot(U(K):U(K)^p)}{(\mathrm{Norm}_{N\vert K}U(N):U(K)^p)}\),
since no real Archimedean place of \(K\) becomes complex in \(N\).
Consequently,
\((E(N\vert K):U(N)^{1-\sigma})=p\cdot(U(K):\mathrm{Norm}_{N\vert K}U(N))\).

\begin{enumerate}

\item
For a class \(c\in\mathrm{Cl}(K)\),
the extension class \(\mathrm{j}_{N\vert K}(c)\in\mathrm{Cl}(N)\)
is invariant under \(\mathcal{G}\), whence
\(\mathrm{Norm}_{N\vert K}(\mathrm{j}_{N\vert K}(c))=c^{\sum_{\ell=1}^p\,\sigma^{\ell-1}}=c^p\).
In particular, we have
\(c^p=\mathrm{Norm}_{N\vert K}(\mathrm{j}_{N\vert K}(c))=\mathrm{Norm}_{N\vert K}(1)=1\)
for \(c\in\mathrm{Ker}(\mathrm{j}_{N\vert K})\).
Independently from the ramification, the principalisation kernel
\(\mathrm{Ker}(\mathrm{j}_{N\vert K})\)
is therefore always contained in the \(p\)-elementary class group of \(K\).

\item
For a complex quadratic field \(K\) and
a relatively unramified dihedral field \(N\) we have
\((U(K):\mathrm{Norm}_{N\vert K}U(N))=1\),
except in the case \(p=3\), \(d(K)=-3\) with
\((U(K):\mathrm{Norm}_{N\vert K}U(N))\in\lbrace 1,3\rbrace\).
But \(K=\mathbb{Q}(\sqrt{-3})\) has class number \(\mathrm{h}(K)=1\)
and therefore does not possess any unramified extensions.

\item
By proposition \ref{p:QdrCoh}, a non-trivial norm index of unit groups
\((U(K):\mathrm{Norm}_{N\vert K}U(N))=p\)
is only possible for a real quadratic field \(K\)
and a relatively unramified dihedral field \(N\) of type \(\alpha\),
whereas \((U(K):\mathrm{Norm}_{N\vert K}U(N))=1\) for \(N\) of type \(\delta\).
\end{enumerate}
\end{proof}

\begin{remark}
\label{r:HrbQuot}
Since the \((-1)\)st cohomology group \(\mathrm{H}^{-1}(\mathrm{Gal}(N\vert K),U(N))\)
coincides with the factor group \(E(N\vert K)/U(N)^{1-\sigma}\),
the equation
\((E(N\vert K):U(N)^{1-\sigma})=p\cdot(U(K):\mathrm{Norm}_{N\vert K}U(N))\)
means that the \textit{Herbrand quotient} \cite[p.92, Th.3]{Hb} of the unit group \(U(N)\) is given by
\[\frac{\lvert\mathrm{H}^0(\mathrm{Gal}(N\vert K),U(N))\rvert}{\lvert\mathrm{H}^{-1}(\mathrm{Gal}(N\vert K),U(N))\rvert}
=\frac{(\mathrm{Ker}(\Delta):\mathrm{Im}(\mathcal{N}))}{(\mathrm{Ker}(\mathcal{N}):\mathrm{Im}(\Delta))}
=\frac{1}{\lbrack N:K\rbrack},\]
where the endomorphisms
\(\Delta:E\mapsto E^{1-\sigma}\)
and
\(\mathcal{N}:E\mapsto E^{\sum_{j=1}^p\,\sigma^{j-1}}\)
of \(U(N)\)
satisfy the relations \(\Delta\circ\mathcal{N}=\mathcal{N}\circ\Delta=1\).
\end{remark}


\subsection{Class number relations}
\label{ss:QdrCln}

Now we assume that the quadratic base field \(K\)
has an elementary abelian bicyclic \(p\)-class group of type \((p,p)\).

\begin{remark}
\label{r:DscMlt}
Since \(K\) is of \(p\)-class rank two,
\cite[p.838, Cor.3.1]{Ma2} implies that
there exist exactly \(\frac{p^2-1}{p-1}=p+1\) non-isomorphic
unramified cyclic extension fields \(N\) of \(K\) of relative degree \(p\)
which share the same discriminant \(d(N)=d(K)^p\).
\end{remark}

Hence, let \(N\) be one of the unramified cyclic extensions
of relative degree \(p\) of \(K\).
By \(L\) we denote the non-Galois subfield of \(N\)
of absolute degree \(p\) over \(\mathbb{Q}\),
which is fixed by the selected generating automorphism \(\tau\) of order \(2\)
of the dihedral group \(\mathrm{Gal}(N\vert\mathbb{Q})\).
\(L\) is one of \(p\) conjugate and thus isomorphic subfields of \(N\).

\begin{proposition}
\label{p:QdrCln}

Let \(p\ge 3\) be an odd prime,
\(K\) a quadratic base field with \(p\)-class group of type \((p,p)\),
and \(N\) an absolutely dihedral unramified extension field of \(K\) of relative degree \(p\)
with non-Galois subfield \(L\) of absolute degree \(p\).

\begin{enumerate}

\item
If \(K\) is a complex quadratic field,
then the \(p\)-class numbers of \(N\) and \(L\) satisfy the relation\\
\(\mathrm{h}_p(N)=p\cdot\mathrm{h}_p(L)^2\), in particular the \(p\)-exponent of \(\mathrm{h}_p(N)\) is odd.

\item
If \(K\) is a real quadratic field, then\\
\(\mathrm{h}_p(N)=p\cdot\mathrm{h}_p(L)^2\) with odd \(p\)-exponent, if \(N\) is of type \(\delta\), and\\
\(\mathrm{h}_p(N)=\mathrm{h}_p(L)^2\) with even \(p\)-exponent, if \(N\) is of type \(\alpha\).

\end{enumerate}

\end{proposition}

\begin{proof}

Provided, that the discriminant \(d(K)\) is different from \(-3\) in the case \(p=3\),
we have the index formula
\[(U(N):E_0U(K))=\frac{p^{r-1}(E(N\vert K):E_0)}{(U(K):\mathrm{Norm}_{N\vert K}U(N))},\]
according to Schmithals \cite[p.53, eq.(6)]{Sm},
where \(r=1\) for a complex
and \(r=2\) for a real quadratic field \(K\).
The relation between the class numbers of the dihedral field \(N\),
its non-Galois subfield \(L\),
and the quadratic base field \(K\),
\[\mathrm{h}(N)=\frac{(U(N):E_0U(K))}{p^r}\mathrm{h}(K)\mathrm{h}(L)^2,\]
is due to Scholz \cite[p.213, p.216]{So} for \(p=3\)
and to Moser \cite[p.67, Th.IV.1]{Mo} for arbitrary odd primes \(p\).
In the unramified case, we have
\[\frac{(U(N):E_0U(K))}{p^r}=\frac{(E(N\vert K):E_0)}{p\cdot(U(K):\mathrm{Norm}_{N\vert K}U(N))},\]
where \((E(N\vert K):E_0)=1\).
With \(\mathrm{h}_p(K)=p^2\),
the \(p\)-contribution of the class numbers is therefore given by\\
\(\mathrm{h}_p(N)=\frac{1}{p}\cdot p^2\cdot\mathrm{h}_p(L)^2\),
if \(K\) is complex, or real with \(N\) of type \(\delta\), and\\
\(\mathrm{h}_p(N)=\frac{1}{p^2}\cdot p^2\cdot\mathrm{h}_p(L)^2\),
if \(K\) is real with \(N\) of type \(\alpha\).
\end{proof}

\begin{remark}
\label{r:QdrCtm}
The case of the quadratic discriminant \(d(K)=-3\),
which had to be excluded repeatedly for \(p=3\),
concerns the cyclotomic quadratic field \(K=\mathbb{Q}(\sqrt{-3})\)
of the third roots of unity.
But since this field has the class number \(\mathrm{h}(K)=1\) and therefore
does not possess unramified cyclic extensions \(N\vert K\) of prime degree,
it is inessential for the present investigations.
\end{remark}


\subsection{Class numbers of the non-Galois subfields}
\label{ss:QdrDir}

Based on the preceding discussion of
the Galois cohomology of the unit group \(U(N)\),
the principalisation of \(p\)-classes of \(K\) in \(N\),
and the parity of the \(p\)-exponent of the \(p\)-class number \(\mathrm{h}_p(N)\)
of a dihedral field \(N\) of degree \(2p\) with an odd prime \(p\ge 3\),
we are now in the position to determine
the \(p\)-class numbers of the non-Galois subfields \(L_1,\ldots,L_{p+1}\)
of the unramified cyclic extensions \(N_1,\ldots,N_{p+1}\) of \(K\)
from a given structure of the second \(p\)-class group \(G=\mathrm{Gal}(\mathrm{F}_p^2(K)\vert K)\)
of a quadratic base field \(K\) with \(p\)-class group of type \((p,p)\).
We begin by considering a group \(G\) of maximal class.

\begin{theorem}
\label{t:QdrMaxDir}
Let \(p\ge 3\) be an odd prime and
\(K\) a quadratic base field
with \(p\)-class group \(\mathrm{Cl}_p(K)\) of type \((p,p)\).
Suppose that the second \(p\)-class group \(G=\mathrm{Gal}(\mathrm{F}_p^2(K)\vert K)\)
of \(K\) is abelian or metabelian of coclass \(\mathrm{cc}(G)=1\)
with order \(\lvert G\rvert=p^m\) and class \(\mathrm{cl}(G)=m-1\),
where \(m\ge 2\).
Assume that the generators of \(G\) satisfy
the conditions of theorem \ref{t:MaxCfg}
and that \(\lbrack\chi_2(G),\gamma_2(G)\rbrack=\gamma_{m-k}(G)\)
with \(k=0\) for \(m\le 3\) and \(0\le k\le m-4\) for \(m\ge 4\).
Then the following statements hold.

\begin{enumerate}

\item
The case of an abelian second \(p\)-class group \(G\),
that is \(m=2\),
is impossible for a quadratic base field \(K\).

\item
In the case of a metabelian second \(p\)-class group \(G\)
of maximal class,
that is \(m\ge 3\),
\(K\) must be a real quadratic field.

\item
\begin{eqnarray*}
L_1&\text{ is of type }&
\begin{cases}
\alpha,&\text{ if }m-k\equiv 1\pmod{2},\ k\ge 0,\\
\delta,&\text{ if }m\equiv 0\pmod{2},\ k=0,\\
\end{cases}\\
L_i&\text{ is of type }&\alpha\text{ for }2\le i\le p+1.
\end{eqnarray*}

\item
The orders of the \(p\)-class groups \(\mathrm{Cl}_p(L_i)\)
of the \(p+1\) totally real non-Galois subfields
\(L_1,\ldots,L_{p+1}\) of absolute degree \(p\)
of the absolutely dihedral unramified extension fields \(N_1,\ldots,N_{p+1}\)
of relative degree \(p\) of \(K\)
are given by

\begin{eqnarray*}
\mathrm{h}_p(L_1)&=&
\begin{cases}
p^{\frac{m-k-1}{2}},&\text{ if }L_1\text{ is of type }\alpha,\ m-k\equiv 1\pmod{2},\ k\ge 0,\\
p^{\frac{m-2}{2}},&\text{ if }L_1\text{ is of type }\delta,\ m\equiv 0\pmod{2},\ k=0,\\
\end{cases}\\
\mathrm{h}_p(L_i)&=&p\text{ for }2\le i\le p+1.
\end{eqnarray*}

\end{enumerate}

\end{theorem}

\begin{proof}

We specialise the statements of theorem \ref{t:MaxDir}
for a quadratic base field \(K=\mathbb{Q}(\sqrt{D})\),
using the propositions \ref{p:QdrPrc} and \ref{p:QdrCln}

\begin{enumerate}

\item
In our paper \cite[Th.2.4]{Ma3} we show that
in the case of an abelian second \(p\)-class group \(G\),
the entire \(p\)-class group \(\mathrm{Cl}_p(K)\)
of an arbitrary base field \(K\)
becomes principal
in all \(p+1\) extension fields \(N_1,\ldots,N_{p+1}\).
Consequently,
the dihedral fields \(N_1,\ldots,N_{p+1}\) are all of type \(\alpha\),
by proposition \ref{p:QdrPrc}, if \(K\) is a quadratic base field.
But, in view of proposition \ref{p:QdrCln}, this is a contradiction
to the fact that
the \(p\)-exponent \(1\) of the \(p\)-class numbers \(\mathrm{h}_p(N_i)=p\)
is odd for \(1\le i\le p+1\), by theorem \ref{t:MaxDir}.

\item
Since the \(p\)-exponent \(2\) of the \(p\)-class numbers \(\mathrm{h}_p(N_i)=p^2\)
is even for \(2\le i\le p+1\), by theorem \ref{t:MaxDir},
the base field \(K\) must be real quadratic, by proposition \ref{p:QdrCln},
and the dihedral fields \(N_2,\ldots,N_{p+1}\) are necessarily of type \(\alpha\).

\item
The \(p\)-exponent \(m-k-1\) of the distinguished \(p\)-class number
\(\mathrm{h}_p(N_1)=p^{m-k-1}\) of theorem \ref{t:MaxDir}
must be odd, by proposition \ref{p:QdrCln},
if the dihedral field \(N_1\) is of type \(\delta\),
and even, if \(N_1\) is of type \(\alpha\).
This yields four possible cases:\\
\(m-1\equiv 1\pmod{2}\) for \(N_1\) of type \(\delta\) and \(k=0\),\\
\(m-1\equiv 0\pmod{2}\) for \(N_1\) of type \(\alpha\) and \(k=0\),\\
\(m-k-1\equiv 1\pmod{2}\) for \(N_1\) of type \(\delta\) and \(k\ge 1\),\\
\(m-k-1\equiv 0\pmod{2}\) for \(N_1\) of type \(\alpha\) and \(k\ge 1\).\\
The last but one case, that
\(N_1\) is of type \(\delta\) and \(k\ge 1\), is impossible,
because in our paper \cite[Th.2.5]{Ma3} we show that
for a second \(p\)-class group \(G\) of coclass \(\mathrm{cc}(G)=1\) with invariant \(k\ge 1\),
the entire \(p\)-class group \(\mathrm{Cl}_p(K)\)
of an arbitrary base field \(K\)
becomes principal
in all \(p+1\) extension fields \(N_1,\ldots,N_{p+1}\).
Consequently,
all dihedral fields \(N_1,\ldots,N_{p+1}\) must be of type \(\alpha\),
by proposition \ref{p:QdrPrc}, if \(K\) is a quadratic base field.

\item
By proposition \ref{p:QdrCln},
the expressions for the \(p\)-class numbers
of the non-Galois subfields \(L_i\) of absolute degree \(p\)
of the dihedral fields \(N_i\) are given,\\
on the one hand, by
\(\mathrm{h}_p(L_i)^2=\frac{\mathrm{h}_p(N_i)}{p}\), if \(N_i\) is of type \(\delta\) and thus \(i=1\),\\
and therefore, by theorem \ref{t:MaxDir},\\
\(\mathrm{h}_p(L_1)=p^{\frac{m-1-1}{2}}=p^{\frac{m-2}{2}}\), since \(k=0\),\\
and on the other hand, by
\(\mathrm{h}_p(L_i)^2=\mathrm{h}_p(N_i)\), if \(N_i\) is of type \(\alpha\),\\
and therefore, by theorem \ref{t:MaxDir},\\
\(\mathrm{h}_p(L_1)=p^{\frac{m-1}{2}}\), if \(k=0\),\\
\(\mathrm{h}_p(L_1)=p^{\frac{m-k-1}{2}}\), if \(k\ge 1\), and\\
\(\mathrm{h}_p(L_i)=p^{\frac{2}{2}}=p\) for \(2\le i\le p+1\).
\end{enumerate}
\end{proof}

We should point out that,
due to the arithmetical properties of dihedral fields,
the assumption of a quadratic base field \(K=\mathbb{Q}(\sqrt{D})\)
has strong consequences for the principalisation over \(K\),
which can be derived for an arbitrary base field \(K\)
only by the computation of the kernels of the
transfers from the second \(p\)-class group \(G=\mathrm{Gal}(\mathrm{F}_p^2(K)\vert K)\)
to its maximal normal subgroups \(M_i\) \cite{Ma3}.
For this purpose we introduce an isomorphism invariant \(\nu=\nu(K)\)
of an arbitrary base field \(K\) with elementary abelian bicyclic \(p\)-class group,
which measures the extent of total principalisation.

\begin{definition}
\label{d:TotPrc}
Let \(0\le\nu\le p+1\) be the number
\(\nu=\#\lbrace 1\le i\le p+1\mid\mathrm{Ker}(\mathrm{j}_{N_i\vert K})=\mathrm{Cl}_p(K)\rbrace\)
of unramified cyclic extensions \(N_i\) of \(K\) of relative degree \(p\),
in which the entire \(p\)-class group \(\mathrm{Cl}_p(K)\) of \(K\)
becomes principal (cfr. \cite[capitulation  number, p.1230]{ChFt}.
\end{definition}

\begin{corollary}
\label{c:QdrMaxDir}

Let \(p\ge 3\) be an odd prime and
\(K\) a quadratic base field
with \(p\)-class group \(\mathrm{Cl}_p(K)\) of type \((p,p)\).
If the second \(p\)-class group \(G=\mathrm{Gal}(\mathrm{F}_p^2(K)\vert K)\)
of \(K\) is metabelian of coclass \(\mathrm{cc}(G)=1\),
then the invariant \(\nu\) of \(K\)
is restricted to the values \(p\le\nu\le p+1\).

\end{corollary}

\begin{proof}
This statement is an immediate consequence of those parts of theorem \ref{t:QdrMaxDir},
which are independent from \cite{Ma3},
together with proposition \ref{p:QdrPrc}.
It does not depend on the values of the invariants \(m\) and \(k\) of the group \(G\).
\end{proof}

However, to prove the stronger result that \(\nu=p+1\)
for \(m\ge 5\) and \(\lbrack\chi_2(G),\gamma_2(G)\rbrack=\gamma_{m-k}(G)\) with \(k\ge 1\),
the theory of the transfers of the group \(G\) must be taken into consideration,
according to our paper \cite{Ma3}.
A further corollary generalises a result of \cite[p.34, A]{SoTa} for \(p=3\).

\begin{corollary}
\label{c:QdrExtSpc}

Let \(p\ge 3\) be an odd prime and
\(K\) a quadratic base field
with \(p\)-class group \(\mathrm{Cl}_p(K)\) of type \((p,p)\).
Then the extra special \(p\)-group \(G_0^{(3)}(0,1)\)
of exponent \(p^2\) and order \(p^3\) cannot occur as
the second \(p\)-class group \(G=\mathrm{Gal}(\mathrm{F}_p^2(K)\vert K)\) of \(K\).

\end{corollary}

\begin{proof}
According to \cite[Th.2.5]{Ma3},
the extra special \(p\)-group \(G=G_0^{(3)}(0,1)\)
is a metabelian \(p\)-group of coclass \(\mathrm{cc}(G)=1\)
with invariant \(\nu=0\),
since it is the unique group with principalisation type \(\mathrm{A}.1\),
for which all transfer kernels coincide and are cyclic of order \(p\).
However, if a quadratic base field \(K\) has
a second \(p\)-class group \(G\) of coclass \(\mathrm{cc}(G)=1\),
then its invariant is restricted to \(p\le\nu\le p+1\),
by corollary \ref{c:QdrMaxDir}.
\end{proof}

Now we turn to the metabelian \(3\)-groups
\(G=\mathrm{Gal}(\mathrm{F}_3^2(K)\vert K)\)
of coclass \(\mathrm{cc}(G)\ge 2\).
It is adequate to treat complex and real quadratic base fields separately,
since only for the latter we have to consider the Galois cohomology.

\begin{theorem}
\label{t:CmpLowDir}

Let \(K\) be a complex quadratic base field
with \(3\)-class group \(\mathrm{Cl}_3(K)\) of type \((3,3)\).
Assume that the second \(3\)-class group \(G=\mathrm{Gal}(\mathrm{F}_3^2(K)\vert K)\)
of \(K\) is metabelian of coclass \(\mathrm{cc}(G)\ge 2\)
with order \(\lvert G\rvert=3^n\), class \(\mathrm{cl}(G)=m-1\), and \(e=n-m+2\),
where \(4\le m<n\le 2m-3\).
Suppose that the generators of \(G\) satisfy
the conditions of theorem \ref{t:LowCfg}
and that \(\lbrack\chi_s(G),\gamma_e(G)\rbrack=\gamma_{m-k}(G)\) with \(0\le k\le 1\).
Then the following statements hold.

\begin{enumerate}

\item
The invariants \(e\) and \(m\) of the group \(G\) satisfy the conditions

\begin{eqnarray*}
e&\equiv&1\pmod{2},\\
m&\equiv&
\begin{cases}
0\pmod{2},&\text{ if }k=0,\\
1\pmod{2},&\text{ if }k=1.\\
\end{cases}
\end{eqnarray*}

\item
The \(3\)-class groups
of the four non-Galois complex cubic subfields \(L_1,\ldots,L_4\)
of the unramified cyclic cubic extension fields \(N_1,\ldots,N_4\) of \(K\)
are cyclic and their order is given by

\begin{eqnarray*}
\mathrm{h}_3(L_1)&=&
\begin{cases}
3^{\frac{m-2}{2}},&\text{ if }k=0,\\
3^{\frac{m-3}{2}},&\text{ if }k=1,\\
\end{cases}\\
\mathrm{h}_3(L_2)&=&3^{\frac{e-1}{2}},\\
\mathrm{h}_3(L_i)&=&3\text{ for }3\le i\le 4.
\end{eqnarray*}

\end{enumerate}

\end{theorem}

\begin{proof}

In the case of a complex quadratic base field \(K\),
proposition \ref{p:QdrCln} states that
the \(3\)-exponent of each of the \(3\)-class numbers \(\mathrm{h}_3(N_i)\)
with \(1\le i\le 4\)
in theorem \ref{t:LowDir} must be odd.

\begin{enumerate}

\item
For the first extension \(N_1\), we therefore obtain the condition\\
\(m-1\equiv 1\pmod{2},\text{ if }k=0,\text{ and }
m-2\equiv 1\pmod{2},\text{ if }k=1\).\\
For the second extension \(N_2\), it follows that
\(e\equiv 1\pmod{2}\).\\
For the third and fourth extension \(N_3,N_4\),
the \(3\)-exponent \(3\) of \(\mathrm{h}_3(N_i)=3^3\)
is odd, a priori.

\item
By proposition \ref{p:QdrCln}, the relation between the \(3\)-class numbers
of \(L_i\) and \(N_i\) is given by
\(\mathrm{h}_3(L_i)^2=\frac{\mathrm{h}_3(N_i)}{3}\).
With the aid of theorem \ref{t:LowDir}, it follows that\\
\(\mathrm{h}_3(L_1)=3^{\frac{m-1-1}{2}}\), if \(k=0\),\\
\(\mathrm{h}_3(L_1)=3^{\frac{m-2-1}{2}}\), if \(k=1\),\\
\(\mathrm{h}_3(L_2)=3^{\frac{e-1}{2}}\), and\\
\(\mathrm{h}_3(L_i)=3^{\frac{3-1}{2}}=3\) for \(3\le i\le 4\).\\
According to Gerth \cite[p.315, Th.3.4]{Ge},
the \(3\)-class ranks of \(L_i\) and \(K\) satisfy the condition
\(\mathrm{r}_3(L_i)=\mathrm{r}_3(K)-1\),
if \(N_i\vert K\) is unramified.
Consequently, \(\mathrm{Cl}_3(L_i)\) is cyclic,
since \(\mathrm{Cl}_3(K)\) is bicyclic.
\end{enumerate}
\end{proof}

\begin{theorem}
\label{t:RealLowDir}

Let \(K\) be a real quadratic base field
with \(3\)-class group \(\mathrm{Cl}_3(K)\) of type \((3,3)\).
Assume that
the second \(3\)-class group \(G=\mathrm{Gal}(\mathrm{F}_3^2(K)\vert K)\)
of \(K\) is metabelian of coclass \(\mathrm{cc}(G)\ge 2\)
with order \(\lvert G\rvert=3^n\), class \(\mathrm{cl}(G)=m-1\), and \(e=n-m+2\),
where \(4\le m<n\le 2m-3\).
Suppose that the generators of \(G\) satisfy
the conditions of theorem \ref{t:LowCfg}
and that \(\lbrack\chi_s(G),\gamma_e(G)\rbrack=\gamma_{m-k}(G)\) with \(0\le k\le 1\).
Then the following statements hold.

\begin{enumerate}

\item

\begin{eqnarray*}
L_1&\text{ is of type }&
\begin{cases}
\alpha,&\text{ if }m\equiv 1(2),\ k=0
\text{ or }m\equiv 0(2),\ k=1,\\
\delta,&\text{ if }m\equiv 0(2),\ k=0
\text{ or }m\equiv 1(2),\ k=1,\\
\end{cases}\\
L_2&\text{ is of type }&
\begin{cases}
\alpha,&\text{ if }e\equiv 0\pmod{2},\\
\delta,&\text{ if }e\equiv 1\pmod{2},\\
\end{cases}\\
L_i&\text{ is of type }&\delta\text{ for }3\le i\le 4.
\end{eqnarray*}

\item
The \(3\)-class groups of the four non-Galois totally real cubic subfields
\(L_1,\ldots,L_4\) of the unramified cyclic cubic field extensions \(N_1,\ldots,N_4\) of \(K\)
are cyclic and their order is given by

\begin{eqnarray*}
\mathrm{h}_3(L_1)&=&
\begin{cases}
3^{\frac{m-1}{2}},&\text{ if }k=0\text{ and }L_1\text{ is of type }\alpha,\\
3^{\frac{m-2}{2}},&\text{ if }k=1\text{ and }L_1\text{ is of type }\alpha
\text{ or }k=0\text{ and }L_1\text{ is of type }\delta,\\
3^{\frac{m-3}{2}},&\text{ if }k=1\text{ and }L_1\text{ is of type }\delta,
\end{cases}\\
\mathrm{h}_3(L_2)&=&
\begin{cases}
3^{\frac{e}{2}},&\text{ if }L_2\text{ is of type }\alpha,\\
3^{\frac{e-1}{2}},&\text{ if }L_2\text{ is of type }\delta,\\
\end{cases}\\
\mathrm{h}_3(L_i)&=&3\text{ for }3\le i\le 4.
\end{eqnarray*}

\end{enumerate}

\end{theorem}

\begin{proof}

In the case of a real quadratic base field \(K\),
proposition \ref{p:QdrCln} states that
the \(3\)-exponent of a \(3\)-class number \(\mathrm{h}_3(N_i)\)
with \(1\le i\le 4\)
in theorem \ref{t:LowDir} must be odd,
if \(N_i\) is of type \(\delta\),
and even,
if \(N_i\) is of type \(\alpha\).

\begin{enumerate}

\item
This yields four possibilities for the first extension \(N_1\),\\
\(m-1\equiv 1(2)\text{ for }k=0,\text{ and }
m-2\equiv 1(2)\text{ for }k=1\),
if \(N_1\) is of type \(\delta\), and\\
\(m-1\equiv 0(2)\text{ for }k=0,\text{ and }
m-2\equiv 0(2)\text{ for }k=1\),
if \(N_1\) is of type \(\alpha\).\\
For the second extension \(N_2\), it follows that\\
\(e\equiv 1\pmod{2}\), if \(N_2\) is of type \(\delta\), and\\
\(e\equiv 0\pmod{2}\), if \(N_2\) is of type \(\alpha\).\\
For the third and fourth extension \(N_3,N_4\),
the \(3\)-exponent \(3\) of \(\mathrm{h}_3(N_i)=3^3\)
is odd, a priori.
Therefore \(N_3\) and \(N_4\) must be of type \(\delta\).

\item
By proposition \ref{p:QdrCln}, the relation between the \(3\)-class numbers
of \(L_i\) and \(N_i\) is given by
\(\mathrm{h}_3(L_i)^2=\frac{\mathrm{h}_3(N_i)}{3}\),
if \(L_i\) is of type \(\delta\), and
\(\mathrm{h}_3(L_i)^2=\mathrm{h}_3(N_i)\),
if \(L_i\)is of type \(\alpha\).\\
With the aid of theorem \ref{t:LowDir}, it follows that\\
\(\mathrm{h}_3(L_1)=3^{\frac{m-1-1}{2}}\) for \(k=0\), and
\(\mathrm{h}_3(L_1)=3^{\frac{m-2-1}{2}}\) for \(k=1\), if \(L_1\) is of type \(\delta\), and\\
\(\mathrm{h}_3(L_1)=3^{\frac{m-1}{2}}\) for \(k=0\), and
\(\mathrm{h}_3(L_1)=3^{\frac{m-2}{2}}\) for \(k=1\), if \(L_1\) is of type \(\alpha\);\\
\(\mathrm{h}_3(L_2)=3^{\frac{e-1}{2}}\), if \(L_2\) is of type \(\delta\), and
\(\mathrm{h}_3(L_2)=3^{\frac{e}{2}}\), if \(L_2\) is of type \(\alpha\);\\
\(\mathrm{h}_3(L_i)=3^{\frac{3-1}{2}}=3\) for \(3\le i\le 4\).\\
According to Gerth \cite[p.315, Th.3.4]{Ge},
the \(3\)-class group \(\mathrm{Cl}_3(L_i)\) must be cyclic,
since \(\mathrm{Cl}_3(K)\) is bicyclic
and \(N_i\) is unramified over \(K\).
\end{enumerate}
\end{proof}

\begin{corollary}
\label{c:QdrLowDir}

Let \(K\) be a quadratic base field
with \(3\)-class group \(\mathrm{Cl}_3(K)\) of type \((3,3)\).
If the second \(3\)-class group \(G=\mathrm{Gal}(\mathrm{F}_3^2(K)\vert K)\)
of \(K\) is metabelian of coclass \(\mathrm{cc}(G)\ge 2\),
then the invariant \(\nu\) of \(K\)
is restricted to the values \(0\le\nu\le 2\), if \(K\) is real,
and to \(\nu=0\), if \(K\) is complex.

\end{corollary}

\begin{proof}

These statements are immediate consequences of
theorem \ref{t:RealLowDir}
together with proposition \ref{p:QdrPrc},
independently from the values of the invariants \(m,n,e\), and \(k\) of the group \(G\).
\end{proof}


\section{Second \(3\)-class groups of quadratic fields of type \((3,3)\)}
\label{s:InvThmCub}

In this section we develop the methods
for the computation of the structure
of the second \(3\)-class group
\(G=\mathrm{Gal}(\mathrm{F}_3^2(K)\vert K)\)
of a quadratic base field \(K=\mathbb{Q}(\sqrt{D})\)
with \(3\)-class group of type \((3,3)\).

The \(3\)-class numbers \(\mathrm{h}_3(L_i)\)
of the four non-Galois absolutely cubic subfields \(L_1,\ldots,L_4\)
of the unramified cyclic cubic extension fields \(N_1,\ldots,N_4\) of \(K\)
determine the order \(\lvert G\rvert=3^n\), class \(\mathrm{cl}(G)=m-1\),
coclass \(\mathrm{cc}(G)=n-m+1\),
and the invariant \(e=n-m+2\) of the second \(3\)-class group
\(G=\mathrm{Gal}(\mathrm{F}_3^2(K)\vert K)\) of \(K\),
if the value \(k=0\) is enforced by the principalisation type of \(K\),
such as for the types
of the sections D,E,F with invariant \(\nu=0\) by Scholz and Taussky \cite{SoTa}
and of the sections c,d with \(\nu=1\) by Nebelung \cite{Ne}.

However, if the invariant \(0\le k\le 1\) is not determined uniquely
by the principalisation type of \(K\) in \(N_1,\ldots,N_4\),
such as for the types of the sections G,H with \(\nu=0\) in \cite{SoTa}
and the sections a,b with \(2\le\nu\le 4\) in \cite{Ne},
then we have to compute the \(3\)-class number \(\mathrm{h}_3(\mathrm{F}_3^1(K))\)
of the first Hilbert \(3\)-class field of \(K\) additionally
to obtain the value of the invariant \(k\).

The assumptions of the following inverse theorems
\ref{t:QdrMaxInv}, \ref{t:QdrLowInv}, and \ref{t:RealLowInv}
are motivated by the results of the theorems
\ref{t:QdrMaxDir}, \ref{t:CmpLowDir}, and \ref{t:RealLowDir}.


\subsection{Dominating total principalisation}
\label{ss:QdrMaxInv}

The first inverse theorem characterises
the real quadratic base fields having the invariant \(3\le\nu\le 4\),
where the second \(3\)-class group
\(G=\mathrm{Gal}(\mathrm{F}_3^2(K)\vert K)\) must be of coclass \(\mathrm{cc}(G)=1\).

\begin{theorem}
\label{t:QdrMaxInv}

Let \(K\) be a real quadratic field
with elementary abelian bicyclic \(3\)-class group.
Suppose that at least three of the
four non-Galois totally real absolutely cubic subfields \(L_1,\ldots,L_4\)
of the absolutely dihedral unramified cyclic cubic extension fields \(N_1,\ldots,N_4\) of \(K\)
are of type \(\alpha\), say \(L_2,L_3,L_4\)
with \(3\)-class numbers \(\mathrm{h}_3(L_2)=\mathrm{h}_3(L_3)=\mathrm{h}_3(L_4)=3\).
Assume that the remaining absolutely cubic field \(L_1\)
has the \(3\)-class number \(\mathrm{h}_3(L_1)=3^u\) with exponent \(u\ge 1\).
Then the second \(3\)-class group
\(G=\mathrm{Gal}(\mathrm{F}_3^2(K)\vert K)\) of \(K\)
is of coclass \(\mathrm{cc}(G)=1\) with \(n=m\), \(e=2\) and the invariants \(m\) and \(k\) are given by
\[
\begin{cases}
m=2u+1\ge 3,\ k=0\text{ or }m=2u+2\ge 6,\ k=1,&\text{ if }L_1\text{ is of type }\alpha,\\
m=2u+2\ge 4,\ k=0,&\text{ if }L_1\text{ is of type }\delta.\\
\end{cases}
\]
In the first case, the invariant \(k=w-2u+1\) is determined by
the \(3\)-class number \(\mathrm{h}_3(\mathrm{F}_3^1(K))=3^{w}\)
of the first Hilbert \(3\)-class field of \(K\),
and \(k=1\) enforces \(u\ge 2\).

\end{theorem}

\begin{proof}

If the totally real cubic fields \(L_2,L_3,L_4\) are of type \(\alpha\), then
the entire \(3\)-class group \(\mathrm{Cl}_3(K)\)
of the real quadratic base field \(K\)
becomes principal
in the three dihedral fields \(N_2,N_3,N_4\),
by proposition \ref{p:QdrPrc}.\\
If the second \(3\)-class group
\(G=\mathrm{Gal}(\mathrm{F}_3^2(K)\vert K)\)
of \(K\) were of coclass \(\mathrm{cc}(G)\ge 2\), then
\(K\) had a total principalisation
in at most two of the four unramified cyclic cubic extensions \(N_i\),
by corollary \ref{c:QdrLowDir}.
Consequently, \(G\) must be of coclass \(\mathrm{cc}(G)=1\) in the present situation.

By theorem \ref{t:QdrMaxDir}, we have\\
\(3^u=\mathrm{h}_3(L_1)=3^{\frac{m-1}{2}}\) with \(m\equiv 1(2)\),
if \(L_1\)is of type \(\alpha\), \(k=0\), but\\
\(3^u=\mathrm{h}_3(L_1)=3^{\frac{m-2}{2}}\) with \(m\equiv 0(2)\),
if \(L_1\) is of type \(\alpha\), \(k=1\), \(m\ge 5\)
or \(L_1\) is of type \(\delta\), \(k=0\).

For the group \(G\), we therefore obtain\\
an odd index of nilpotence \(m=2u+1\ge 3\),
if \(L_1\) is of type \(\alpha\) and \(k=0\), and\\
an even index of nilpotence \(m=2u+2\ge 4\),
if \(L_1\) is of type \(\delta\) and \(k=0\),\\
or \(m=2u+2\ge 6\) with \(u\ge 2\),
if \(L_1\) is of type \(\alpha\) and \(k=1\).

Finally, if \(L_1\) is of type \(\alpha\), the invariant \(k\)
is determined by
\(3^w=\mathrm{h}_3(\mathrm{F}_3^1(K))=\lvert\gamma_2(G)\rvert=3^{m-2}=3^{2u+k-1}\),
that is \(k=w-2u+1\).
\end{proof}


\subsection{Partial principalisation}
\label{ss:QdrLowInv}

By the second inverse theorem we cover all the complex quadratic base fields
and the real quadratic base fields having the invariant \(\nu=0\),
which show a very similar behavior.
Here, the second \(3\)-class group
\(G=\mathrm{Gal}(\mathrm{F}_3^2(K)\vert K)\) is of coclass \(\mathrm{cc}(G)\ge 2\).

\begin{theorem}
\label{t:QdrLowInv}

Let \(K\) be a quadratic field
with elementary abelian bicyclic \(3\)-class group.
In the case of a real quadratic field \(K\), let all 
four absolutely dihedral unramified cyclic cubic extension fields \(N_1,\ldots,N_4\) of \(K\)
be of type \(\delta\).
Suppose that at least two
of the four non-Galois absolutely cubic subfields \(L_1,\ldots,L_4\)
of \(N_1,\ldots,N_4\), say \(L_3,L_4\),
have \(3\)-class numbers \(\mathrm{h}_3(L_3)=\mathrm{h}_3(L_4)=3\).
Assume that the remaining two absolutely cubic fields \(L_1,L_2\) have
\(3\)-class numbers \(\mathrm{h}_3(L_1)=3^u\) and \(\mathrm{h}_3(L_2)=3^v\)
with exponents \(u\ge v\ge 1\).
Then the second \(3\)-class group
\(G=\mathrm{Gal}(\mathrm{F}_3^2(K)\vert K)\) of \(K\)
is of coclass \(\mathrm{cc}(G)\ge 2\) with \(4\le m<n\le 2m-3\), \(e=n-m+2\ge 3\),
the invariants \(m\) and \(n\) are given by
\[
\begin{cases}
m=2u+2\ge 4,\ n=2u+2v+1\ge 5,&\text{ if }k=0,\\
m=2u+3\ge 5,\ n=2u+2v+2\ge 6,&\text{ if }k=1,
\end{cases}
\]
the invariant \(e\) has the odd value \(2v+1\ge 3\), and
the invariant \(k=w-2u-2v+1\) is determined by
the \(3\)-class number \(\mathrm{h}_3(\mathrm{F}_3^1(K))=3^{w}\)
of the first Hilbert \(3\)-class field of \(K\).

\end{theorem}

\begin{proof}

If either \(K\) is complex or if \(K\) is real and
all four totally real cubic fields \(L_1,\ldots,L_4\) are of type \(\delta\),
then in none of the four dihedral fields \(N_1,\ldots,N_4\)
the entire \(3\)-class group \(\mathrm{Cl}_3(K)\)
of the quadratic base field \(K\)
can become principal,
by proposition \ref{p:QdrPrc}.\\
If, however, the second \(3\)-class group
\(G=\mathrm{Gal}(\mathrm{F}_3^2(K)\vert K)\)
of \(K\) were of coclass \(\mathrm{cc}(G)=1\), then
\(K\) had a total principalisation
in at least three of the four unramified cyclic cubic extensions \(N_i\)
by corollary \ref{c:QdrMaxDir}.
Consequently, \(G\) must be of coclass \(\mathrm{cc}(G)\ge 2\), in the present situation.

According to theorem \ref{t:CmpLowDir} for complex quadratic fields
and to theorem \ref{t:RealLowDir} for real quadratic fields with \(\nu=0\),
that is \(L_1,L_2\) of type \(\delta\), we have on the one hand\\
\(3^u=\mathrm{h}_3(L_1)=3^{\frac{m-2-k}{2}}\) with \(0\le k\le 1\)
and on the other hand\\
\(3^v=\mathrm{h}_3(L_2)=3^{\frac{e-1}{2}}\).
For the group \(G\), we therefore obtain\\
an even index of nilpotence \(m=2u+2\ge 4\),
if \(k=0\), and\\
an odd index of nilpotence \(m=2u+3\ge 5\),
if \(k=1\).\\
The invariant \(e=2v+1\) is always odd.\\
Consequently, the \(3\)-exponent \(n=e+m-2\) of the group order is given by\\
\(n=2v+1+2u\) for \(k=0\) and
\(n=2v+1+2u+1\) for \(k=1\).

Finally, the invariant \(k\) is determined by
\(3^w=\mathrm{h}_3(\mathrm{F}_3^1(K))=\lvert\gamma_2(G)\rvert=3^{n-2}=3^{2u+2v+k-1}\),
that is \(k=w-2u-2v+1\).
\end{proof}


\subsection{Mixed principalisation}
\label{ss:RealLowInv}

The third inverse theorem characterises
the real quadratic base fields with invariant \(1\le\nu\le 2\),
where the second \(3\)-class group
\(G=\mathrm{Gal}(\mathrm{F}_3^2(K)\vert K)\) must be of coclass \(\mathrm{cc}(G)\ge 2\).

\begin{theorem}
\label{t:RealLowInv}

Let \(K\) be a real quadratic field
with elementary abelian bicyclic \(3\)-class group.
Suppose that two of the four
non-Galois totally real absolutely cubic subfields \(L_1,\ldots,L_4\)
of the unramified cyclic cubic extension fields \(N_1,\ldots,N_4\) of \(K\),
say \(L_3,L_4\), are of type \(\delta\)
with \(3\)-class numbers \(\mathrm{h}_3(L_3)=\mathrm{h}_3(L_4)=3\).
Assume that the remaining two absolutely cubic fields \(L_1,L_2\)
have \(3\)-class numbers \(\mathrm{h}_3(L_1)=3^u\) and \(\mathrm{h}_3(L_2)=3^v\)
with exponents \(u\ge v\ge 1\) and that
the ordered pair of their types is denoted by
\(T=(T_1,T_2)\in\lbrace(\alpha,\alpha),(\alpha,\delta),(\delta,\alpha)\rbrace\).
Then the second \(3\)-class group
\(G=\mathrm{Gal}(\mathrm{F}_3^2(K)\vert K)\) of \(K\)
is of coclass \(\mathrm{cc}(G)\ge 2\) with \(4\le m<n\le 2m-3\), \(e=n-m+2\ge 3\),
and the invariants \(m,n\), and \(k\) are given by
\[
\begin{cases}
m=2u+1\ge 5,\ n=2u+2v-1\ge 7,\ k=0\text{ or }\\
m=2u+2\ge 6,\ n=2u+2v\ge 8,\ k=1,&\text{ if }T=(\alpha,\alpha),\\
m=2u+1\ge 5,\ n=2u+2v\ge 6,\ k=0,&\text{ if }T=(\alpha,\delta),\\
m=2u+2\ge 6,\ n=2u+2v\ge 8,\ k=0,&\text{ if }T=(\delta,\alpha).
\end{cases}
\]
In the first case, the invariant \(k=w-2u-2v+3\)
is determined by \(3\)-class number \(\mathrm{h}_3(\mathrm{F}_3^1(K))=3^{w}\)
of the first Hilbert \(3\)-class field of \(K\).\\
The invariant \(e\) is given by
\[e=
\begin{cases}
2v\ge 4,&\text{ if }L_2\text{ is of type }\alpha,\text{ and}\\
2v+1\ge 3,&\text{ if }L_2\text{ is of type }\delta.
\end{cases}
\]
In particular,
\(u\ge 2\) generally,
and \(v\ge 2\) for \(L_2\) of type \(\alpha\).

\end{theorem}

\begin{proof}

If \(K\) is a real quadratic base field and
at least two of the totally real cubic fields \(L_1,\ldots,L_4\)
are of type \(\delta\),
then the number of dihedral fields \(N_i\vert K\) with total principalisation
is given by the invariant \(\nu\le 2\), by proposition \ref{p:QdrPrc}.\\
If the second \(3\)-class group of \(K\)
\(G=\mathrm{Gal}(\mathrm{F}_3^2(K)\vert K)\)
were of coclass \(\mathrm{cc}(G)=1\), then
\(K\) had an invariant \(\nu\ge 3\),
by corollary \ref{c:QdrMaxDir}.\\
Hence, \(G\) must be of coclass \(\mathrm{cc}(G)\ge 2\), in the present situation.

By theorem \ref{t:RealLowDir} for real quadratic fields with \(1\le\nu\le 2\),
we have on the one hand
\[3^u=\mathrm{h}_3(L_1)=
\begin{cases}
3^{\frac{m-1}{2}}&\text{ for }L_1\text{ of type }\alpha,\ k=0,\\
3^{\frac{m-2}{2}}&\text{ for }L_1\text{ of type }\alpha,\ k=1,\\
3^{\frac{m-2}{2}}&\text{ for }L_1\text{ of type }\delta,\ k=0,\\
3^{\frac{m-3}{2}}&\text{ for }L_1\text{ of type }\delta,\ k=1,
\end{cases}
\]
and on the other hand
\[3^v=\mathrm{h}_3(L_2)=
\begin{cases}
3^{\frac{e}{2}}&\text{ for }L_2\text{ of type }\alpha,\\
3^{\frac{e-1}{2}}&\text{ for }L_2\text{ of type }\delta.\\
\end{cases}
\]
Since the pair \(T=(\delta,\delta)\) has been treated in theorem \ref{t:QdrLowInv} already,
we must have \(T=(\delta,\alpha)\), if \(L_1\) is of type \(\delta\).
Then the last case, \(k=1\), is impossible for the following reason.
According to \cite[p.208, Satz 6.14, p.189 ff]{Ne},
the principalisation types of the sections c,d,
which correspond to the pairs \(T=(\delta,\alpha)\) and \(T=(\alpha,\delta)\),
can only occur with the value \(k=0\).
For the group \(G\), we therefore obtain\\
an even index of nilpotence \(m=2u+2\ge 4\)
for \(L_1\) of type \(\delta\), \(k=0\) or \(L_1\) of type \(\alpha\), \(k=1\)\\
and an odd index of nilpotence \(m=2u+1\ge 5\) with \(u\ge 2\),
for \(L_1\) of type \(\alpha\), \(k=0\).\\
The invariant \(e\) is
even \(e=2v\ge 4\) with \(v\ge 2\) for \(L_2\) of type \(\alpha\) and\\
odd \(e=2v+1\ge 3\) for \(L_2\) of type \(\delta\).\\
Further, for a group \(G\) of coclass \(\mathrm{cc}(G)\ge 3\) with \(e\ge 4\),
we must have an index of nilpotence \(m\ge 5\), since \(e\le m-1\).
Therefore, the \(3\)-exponent \(n=e+m-2\) of the group order is given by
\[n=
\begin{cases}
2v+2u+1-2\ge 4+5-2=7&\text{ for }T=(\alpha,\alpha),\ k=0,\\
2v+2u+2-2\ge 4+6-2=8&\text{ for }T=(\alpha,\alpha),\ k=1,\\
2v+1+2u+1-2\ge 3+5-2=6&\text{ for }T=(\alpha,\delta),\ k=0,\\
2v+2u+2-2\ge 4+6-2=8&\text{ for }T=(\delta,\alpha),\ k=0.
\end{cases}\]

Finally, in the case \(T=(\alpha,\alpha)\), the invariant \(k\) is determined by
\(3^w=\mathrm{h}_3(\mathrm{F}_3^1(K))=\lvert\gamma_2(G)\rvert=3^{n-2}=3^{2u+2v+k-3}\),
and thus \(k=w-2u-2v+3\).
\end{proof}


\section{Computational results for quadratic fields of type \((3,3)\)}
\label{s:NumRstCub}

In this section we apply the methods of section \ref{s:InvThmCub}
for the computation of the structure
of the second \(3\)-class group
\(G=\mathrm{Gal}(\mathrm{F}_3^2(K)\vert K)\)
of a quadratic base field \(K=\mathbb{Q}(\sqrt{D})\)
with \(3\)-class group of type \((3,3)\)
to the range \(-10^6<D<10^7\) of discriminants
and we summarise the concrete numerical results
of these extensive computations,
which exceed all previous numerical investigations by far.
The history of determining principalisation types is shown in table \ref{tab:PrincipalisationHistory}.

\renewcommand{\arraystretch}{1.2}
\begin{table}[ht]
\caption{History of investigating quadratic fields of type \((3,3)\)}
\label{tab:PrincipalisationHistory}
\begin{center}
\begin{tabular}{|l|c|c|r|c|r|}
\hline
 \multicolumn{2}{|c|}{History}    & \multicolumn{2}{|c|}{complex} & \multicolumn{2}{|c|}{real}    \\
 authors            & references  & range            & number     & range            & number     \\
\hline                            
 Scholz, Taussky    & \cite{SoTa} & \(-10\,000<D<0\) &      \(2\) &                  &            \\
 Heider, Schmithals & \cite{HeSm} & \(-20\,000<D<0\) &     \(13\) & \(0<D<100\,000\) &      \(5\) \\
 Brink              & \cite{Br}   & \(-96\,000<D<0\) &     \(66\) &                  &            \\
 Mayer              & \cite{Ma1}  & \(-30\,000<D<0\) &     \(35\) &                  &            \\
 Mayer              & \cite{Ma}   &                  &            & \(0<D<200\,000\) &     \(16\) \\
 Mayer              & \cite{Ma4}  & \(-10^6<D<0\)    & \(2\,020\) & \(0<D<10^7\)     & \(2\,576\) \\
\hline
\end{tabular}
\end{center}
\end{table}


Among the \(2\,576\) real quadratic number fields \(K=\mathbb{Q}(\sqrt{D})\)
with discriminant \(D<10^7\)
and \(3\)-class group \(\mathrm{Cl}_3(K)\) of type \((3,3)\),
the dominating part of \(2\,303\) fields, that is \(89.4\%\),
has at least a threefold total principalisation in the dihedral fields \(N_1,\ldots,N_4\).
Consequently, the second \(3\)-class group
\(G=\mathrm{Gal}(\mathrm{F}_3^2(K)\vert K)\) is
a vertex on the coclass graph \(\mathcal{G}(3,1)\)
of all \(3\)-groups of coclass \(\mathrm{cc}(G)=1\) \cite{EkFs}.

\begin{example}
\label{ex:QdrMaxInv}

For each principalisation type \(\varkappa\) with invariant \(3\le\nu\le 4\),
table \ref{tab:QdrMaxInv} shows
the minimal discriminant \(D\) and the absolute frequency
of real quadratic number fields \(K=\mathbb{Q}(\sqrt{D})\)
with \(3\)-class group of type \((3,3)\)
in the range \(0<D<10^7\) of discriminants,
whose second  \(3\)-class group \(G=\mathrm{Gal}(\mathrm{F}_3^2(K)\vert K)\)
is of coclass \(\mathrm{cc}(G)=1\)
with index of nilpotence \(m\) and invariant \(k\)
and realises the given type \(\varkappa\).

Every principalisation type \(\varkappa\)
is characterised by a lower case section letter \cite{Ne}
and by an additional digit \cite{Ma3}.
The principalisation types \(\mathrm{a}.3\) and \(\mathrm{a}.3^\ast\)
differ by the structure of the \(3\)-class group
\(\mathrm{Cl}_3(N_1)\) of the first dihedral field \(N_1\),
which is nearly homocyclic of type \((9,3)\) in the first case
and elementary abelian of type \((3,3,3)\) in the last case.

We start with the principalisation type \(\varkappa\),
the corresponding quadruplet of types of the totally real cubic fields \(L_1,\ldots,L_4\),
according to proposition \ref{p:QdrPrc},
and the \(3\)-exponents \(u,w\) of the \(3\)-class numbers
\(\mathrm{h}_3(L_1)=3^u\) and \(\mathrm{h}_3(\mathrm{F}_3^1(K))=3^w\),
which have been computed together with the class group structures
as experimental data by means of PARI/GP \cite{PARI},
and calculate the index of nilpotence \(m\) and the invariant \(k\)
of the second \(3\)-class group \(G\) of \(K\)
by the formulas of theorem \ref{t:QdrMaxInv}:
\(m=2u+2\ge 6\), \(k=w-2u+1=1\), if \(L_1\) is of type \(\alpha\), and
\(m=2u+2\ge 4\), \(k=w-2u=0\), if \(L_1\) is of type \(\delta\).
The case of odd \(m=2u+1\ge 3\) and \(k=w-2u+1=0\)
for the principalisation type \(\mathrm{a}.1\)
does not occur and is probably impossible for quadratic base fields.

Each of these principalisation types,
with the only exception of \(\mathrm{a}.3^\ast\),
can occur with different indices of nilpotence \(m\),
which gives rise to \textit{excited states} indicated by arrows \(\uparrow,\uparrow^2,\ldots\).
Types without references have been unknown, up to now.

\end{example}

\begin{table}[ht]
\caption{Second \(3\)-class groups with invariant \(3\le\nu\le 4\) for \(D>0\)}
\label{tab:QdrMaxInv}
\begin{center}
\begin{tabular}{|lcc|rcr|rc|r|c|lr|}
\hline
 type & \(\varkappa\) & types of \(L_i\)               & \(u\) & \(\mathrm{Cl}_3(\mathrm{F}_3^1(K))\) & \(w\) & \(m\) & \(k\) &      min. \(D\) & ref.        & \multicolumn{2}{|c|}{frequency} \\
\hline
 a.1  & \((0000)\)    & \((\alpha\alpha\alpha\alpha)\) & \(2\) &                            \((9,9)\) & \(4\) & \(6\) & \(1\) &     \(62\,501\) & \cite{HeSm} &  &  \(147\) \\
\cline{11-12}
 a.2  & \((1000)\)    & \((\delta\alpha\alpha\alpha)\) & \(1\) &                            \((3,3)\) & \(2\) & \(4\) & \(0\) &     \(72\,329\) & \cite{HeSm} & \multirow{2}{*}{\(\Bigr\rbrace\)} & \multirow{2}{*}{\(1\,386\)} \\
 a.3  & \((2000)\)    & \((\delta\alpha\alpha\alpha)\) & \(1\) &                            \((3,3)\) & \(2\) & \(4\) & \(0\) &     \(32\,009\) & \cite{HeSm} &&\\
\hline
 a.3* & \((2000)\)    & \((\delta\alpha\alpha\alpha)\) & \(1\) &                            \((3,3)\) & \(2\) & \(4\) & \(0\) &    \(142\,097\) & \cite{Ma}   &  &  \(697\) \\
\hline
 a.1\(\uparrow\)  & \((0000)\)    & \((\alpha\alpha\alpha\alpha)\) & \(3\) &                          \((27,27)\) & \(6\) & \(8\) & \(1\) & \(2\,905\,160\) &             &  &    \(1\) \\
\cline{11-12}
 a.2\(\uparrow\)  & \((1000)\)    & \((\delta\alpha\alpha\alpha)\) & \(2\) &                            \((9,9)\) & \(4\) & \(6\) & \(0\) &    \(790\,085\) &             & \multirow{2}{*}{\(\Bigr\rbrace\)} & \multirow{2}{*}{\(72\)} \\
 a.3\(\uparrow\)  & \((2000)\)    & \((\delta\alpha\alpha\alpha)\) & \(2\) &                            \((9,9)\) & \(4\) & \(6\) & \(0\) &    \(494\,236\) &             &&\\
\hline
\multicolumn{10}{|r|}{total:}																																																																	&  & \(2\,303\) \\
\hline
\end{tabular}
\end{center}
\end{table}

Among the \(2\,576\) real quadratic number fields \(K=\mathbb{Q}(\sqrt{D})\)
with discriminant \(0<D<10^7\)
and \(3\)-class group \(\mathrm{Cl}_3(K)\) of type \((3,3)\),
there is a modest part of \(206\) fields, that is \(8.0\%\),
which do not have a total principalisation in the dihedral fields \(N_1,\ldots,N_4\).
Therefore the second \(3\)-class group
\(G=\mathrm{Gal}(\mathrm{F}_3^2(K)\vert K)\) is
a vertex on one of the coclass graphs \(\mathcal{G}(3,r)\)
with \(r\in\lbrace 2,4,6\rbrace\) \cite{EkLg,DEF}.
The same is true for the entire set of all \(2\,020\)
complex quadratic number fields \(K=\mathbb{Q}(\sqrt{D})\)
with discriminant \(-10^6<D<0\)
and \(3\)-class group \(\mathrm{Cl}_3(K)\) of type \((3,3)\).

\begin{example}
\label{CmpLowInv}

For each principalisation type \(\varkappa\) with invariant \(\nu=0\),
table \ref{tab:CmpLowInv} shows
the minimal value of the discriminant \(\lvert D\rvert\)
and the absolute frequency
of complex quadratic number fields \(K=\mathbb{Q}(\sqrt{D})\)
with \(3\)-class group of type \((3,3)\)
in the range \(-10^6<D<0\) of discriminants,
whose second \(3\)-class group \(G=\mathrm{Gal}(\mathrm{F}_3^2(K)\vert K)\)
is of coclass \(\mathrm{cc}(G)\ge 2\)
with index of nilpotence \(m\) and invariants \(e\) and \(k\)
and realises the given type \(\varkappa\).
Every principalisation type \(\varkappa\) is characterised
by an upper case section letter \cite{SoTa}
and an additional numerical identifier \cite{Ma1,Ma3}.

Starting with the \(3\)-exponents \(u,v,w\) of the \(3\)-class numbers
\(\mathrm{h}_3(L_1)=3^u\), \(\mathrm{h}_3(L_2)=3^v\),
and \(\mathrm{h}_3(\mathrm{F}_3^1(K))=3^w\),
which have been computed together with the class group structures
as experimental data with the aid of PARI/GP \cite{PARI},
we calculate the index of nilpotence \(m\),
the \(3\)-exponent \(n\) of the group order \(\lvert G\rvert=3^n\),
and the invariants \(e\) and \(k\)
of the second \(3\)-class group \(G\) of \(K\)
by means of the formulas
\(m=2u+k+2\), \(n=2u+2v+k+1\), \(e=2v+1\), \(k=w-2u-2v+1\),
according to theorem \ref{t:QdrLowInv}.

Each of these principalisation types,
with the only exception of \(\mathrm{D}.5\) and \(\mathrm{D}.10\),
can appear with different indices of nilpotence \(m\),
the types of sections F,G,H even with different values of the invariant \(e\).
We point out that
the principalisation types \(\mathrm{G}.16\), \(\mathrm{G}.19\), and \(\mathrm{H}.4\)
with odd \(m=7\) and \(n=10=2m-4\) can occur in
a regular variant (r) with \(\mathrm{Cl}_3(\mathrm{F}_3^1(K))\) of type \((27,9,9,3)\)
and an irregular variant (i) with \(\mathrm{Cl}_3(\mathrm{F}_3^1(K))\) of type \((9,9,9,9)\)
\cite[p.131, Satz 4.2.4]{Ne}.

\noindent
Finally, the trailing three principalisation types are associated with the
biggest orders \(3^{12}\) and \(3^{13}\) of second \(3\)-class groups \(G\),
known up to now.
These results realise our suggestion in \cite[p.77, 3]{Ma1}

\end{example}

\begin{table}[ht]
\caption{Second \(3\)-class groups with invariant \(\nu=0\) for \(D<0\)}
\label{tab:CmpLowInv}
\begin{center}
\begin{tabular}{|lc|rrcr|rrrc|r|c|lr|}
\hline
 type & \(\varkappa\) & \(u\) & \(v\) & \(\mathrm{Cl}_3(\mathrm{F}_3^1(K))\) & \(w\) & \(e\) & \(m\) &  \(n\) & \(k\) & min. \(\lvert D\rvert\) & ref.        & \multicolumn{2}{|c|}{frequency} \\
\hline
 D.5  & \((4224)\)    & \(1\) & \(1\) &                          \((3,3,3)\) & \(3\) & \(3\) & \(4\) &  \(5\) & \(0\) &             \(12\,131\) & \cite{HeSm} & & \(269\) \\
 D.10 & \((2241)\)    & \(1\) & \(1\) &                          \((3,3,3)\) & \(3\) & \(3\) & \(4\) &  \(5\) & \(0\) &              \(4\,027\) & \cite{SoTa} & & \(667\) \\
\cline{13-14}
 E.6  & \((1313)\)    & \(2\) & \(1\) &                          \((9,9,3)\) & \(5\) & \(3\) & \(6\) &  \(7\) & \(0\) &             \(15\,544\) & \cite{HeSm} & \multirow{2}{*}{\(\Bigr\rbrace\)} & \multirow{2}{*}{\(186\)} \\
 E.14 & \((2313)\)    & \(2\) & \(1\) &                          \((9,9,3)\) & \(5\) & \(3\) & \(6\) &  \(7\) & \(0\) &             \(16\,627\) & \cite{HeSm} & &         \\
\cline{13-14}
 E.8  & \((1231)\)    & \(2\) & \(1\) &                          \((9,9,3)\) & \(5\) & \(3\) & \(6\) &  \(7\) & \(0\) &             \(34\,867\) &             & \multirow{2}{*}{\(\Bigr\rbrace\)} & \multirow{2}{*}{\(197\)} \\
 E.9  & \((2231)\)    & \(2\) & \(1\) &                          \((9,9,3)\) & \(5\) & \(3\) & \(6\) &  \(7\) & \(0\) &              \(9\,748\) & \cite{SoTa} & &         \\
\cline{13-14}
 F.7  & \((3443)\)    & \(2\) & \(2\) &                        \((9,9,9,3)\) & \(7\) & \(5\) & \(6\) &  \(9\) & \(0\) &            \(124\,363\) &             & \multirow{4}{*}{\Huge\(\Bigr\rbrace\)} & \multirow{4}{*}{\(78\)} \\
 F.11 & \((1143)\)    & \(2\) & \(2\) &                        \((9,9,9,3)\) & \(7\) & \(5\) & \(6\) &  \(9\) & \(0\) &             \(27\,156\) &\cite{Br,Ma1}& &         \\
 F.12 & \((1343)\)    & \(2\) & \(2\) &                        \((9,9,9,3)\) & \(7\) & \(5\) & \(6\) &  \(9\) & \(0\) &             \(31\,908\) & \cite{Br}   & &         \\
 F.13 & \((3143)\)    & \(2\) & \(2\) &                        \((9,9,9,3)\) & \(7\) & \(5\) & \(6\) &  \(9\) & \(0\) &             \(67\,480\) & \cite{Br}   & &         \\
\cline{13-14}
 G.16 & \((4231)\)    & \(2\) & \(1\) &                         \((27,9,3)\) & \(6\) & \(3\) & \(7\) &  \(8\) & \(1\) &             \(17\,131\) & \cite{HeSm} & &  \(79\) \\
 G.19 & \((2143)\)    & \(1\) & \(1\) &                        \((3,3,3,3)\) & \(4\) & \(3\) & \(5\) &  \(6\) & \(1\) &             \(12\,067\) & \cite{HeSm} & &  \(94\) \\
 H.4  & \((4443)\)    & \(1\) & \(1\) &                          \((9,3,3)\) & \(4\) & \(3\) & \(5\) &  \(6\) & \(1\) &              \(3\,896\) & \cite{HeSm} & & \(297\) \\
\hline
 E.6\(\uparrow\)  & \((1313)\)    & \(3\) & \(1\) &                        \((27,27,3)\) & \(7\) & \(3\) & \(8\) &  \(9\) & \(0\) &            \(268\,040\) &             & \multirow{2}{*}{\(\Bigr\rbrace\)} & \multirow{2}{*}{\(15\)} \\
 E.14\(\uparrow\) & \((2313)\)    & \(3\) & \(1\) &                        \((27,27,3)\) & \(7\) & \(3\) & \(8\) &  \(9\) & \(0\) &            \(262\,744\) &             & &         \\
\cline{13-14}
 E.8\(\uparrow\)  & \((1231)\)    & \(3\) & \(1\) &                        \((27,27,3)\) & \(7\) & \(3\) & \(8\) &  \(9\) & \(0\) &            \(370\,740\) &             & \multirow{2}{*}{\(\Bigr\rbrace\)} & \multirow{2}{*}{\(13\)} \\
 E.9\(\uparrow\)  & \((2231)\)    & \(3\) & \(1\) &                        \((27,27,3)\) & \(7\) & \(3\) & \(8\) &  \(9\) & \(0\) &            \(297\,079\) &             & &         \\
\cline{13-14}
 F.7\(\uparrow\)  & \((3443)\)    & \(3\) & \(2\) &                      \((27,27,9,3)\) & \(9\) & \(5\) & \(8\) & \(11\) & \(0\) &            \(469\,816\) &             & \multirow{4}{*}{\Huge\(\Bigr\rbrace\)} & \multirow{4}{*}{\(14\)} \\
 F.11\(\uparrow\) & \((1143)\)    & \(3\) & \(2\) &                      \((27,27,9,3)\) & \(9\) & \(5\) & \(8\) & \(11\) & \(0\) &            \(469\,787\) &             & &         \\
 F.12\(\uparrow\) & \((1343)\)    & \(3\) & \(2\) &                      \((27,27,9,3)\) & \(9\) & \(5\) & \(8\) & \(11\) & \(0\) &            \(249\,371\) &             & &         \\
 F.13\(\uparrow\) & \((3143)\)    & \(3\) & \(2\) &                      \((27,27,9,3)\) & \(9\) & \(5\) & \(8\) & \(11\) & \(0\) &            \(159\,208\) &             & &         \\
\cline{13-14}
 G.16\(\uparrow\) & \((4231)\)    & \(3\) & \(1\) &                        \((81,27,3)\) & \(8\) & \(3\) & \(9\) & \(10\) & \(1\) &            \(819\,743\) &             & &   \(2\) \\
 H.4\(\uparrow\)  & \((3313)\)    & \(2\) & \(1\) &                         \((27,9,3)\) & \(6\) & \(3\) & \(7\) &  \(8\) & \(1\) &             \(21\,668\) &\cite{Br,Ma1}& &  \(63\) \\
 H.4\(\uparrow^2\)  & \((3313)\)    & \(3\) & \(1\) &                        \((81,27,3)\) & \(8\) & \(3\) & \(9\) & \(10\) & \(1\) &            \(446\,788\) &             & &   \(6\) \\
\hline
 G.16r & \((1243)\)    & \(2\) & \(2\) &                       \((27,9,9,3)\) & \(8\) & \(5\) & \(7\) & \(10\) & \(1\) &            \(290\,703\) &             & \multirow{3}{*}{\(\Biggr\rbrace\)} & \multirow{3}{*}{\(19\)} \\
 G.19r & \((2143)\)    & \(2\) & \(2\) &                       \((27,9,9,3)\) & \(8\) & \(5\) & \(7\) & \(10\) & \(1\) &             \(96\,827\) &             & &         \\
 H.4r  & \((3343)\)    & \(2\) & \(2\) &                       \((27,9,9,3)\) & \(8\) & \(5\) & \(7\) & \(10\) & \(1\) &            \(256\,935\) &             & &         \\
\cline{13-14}
 G.16i & \((1243)\)    & \(2\) & \(2\) &                        \((9,9,9,9)\) & \(8\) & \(5\) & \(7\) & \(10\) & \(1\) &            \(135\,059\) &             & \multirow{3}{*}{\(\Biggr\rbrace\)} & \multirow{3}{*}{\(15\)} \\
 G.19i & \((2143)\)    & \(2\) & \(2\) &                        \((9,9,9,9)\) & \(8\) & \(5\) & \(7\) & \(10\) & \(1\) &            \(199\,735\) &             & &         \\
 H.4i  & \((3343)\)    & \(2\) & \(2\) &                        \((9,9,9,9)\) & \(8\) & \(5\) & \(7\) & \(10\) & \(1\) &            \(186\,483\) &             & &         \\
\hline
 F.12\(\uparrow^2\) & \((1343)\)    & \(3\) & \(3\) &                     \((27,27,27,9)\) &\(11\) & \(7\) & \(8\) & \(13\) & \(0\) &            \(423\,640\) &             & &   \(1\) \\
 G.19r\(\uparrow\) & \((2143)\)    & \(3\) & \(2\) &                      \((81,27,9,3)\) &\(10\) & \(5\) & \(9\) & \(12\) & \(1\) &            \(509\,160\) &             & &   \(2\) \\
 H.4r\(\uparrow\)  & \((3343)\)    & \(3\) & \(2\) &                      \((81,27,9,3)\) &\(10\) & \(5\) & \(9\) & \(12\) & \(1\) &            \(678\,804\) &             & &   \(3\) \\
\hline
\multicolumn{12}{|r|}{total:}																																																				 	  &  & \(2\,020\) \\
\hline
\end{tabular}
\end{center}
\end{table}

\begin{example}
\label{RealLowInv}

For each principalisation type \(\varkappa\) with invariant \(\nu=0\),
table \ref{tab:RealLowInv} shows
the minimal discriminant \(D\) and the absolute frequency of
real quadratic number fields \(K=\mathbb{Q}(\sqrt{D})\)
with \(3\)-class group of type \((3,3)\)
in the range \(0<D<10^7\) of discriminants,
whose second \(3\)-class group \(G=\mathrm{Gal}(\mathrm{F}_3^2(K)\vert K)\)
is of coclass \(\mathrm{cc}(G)\ge 2\)
with index of nilpotence \(m\) and invariants \(e\) and \(k\)
and realises the given type \(\varkappa\).
The table entries correspond to those of table \ref{tab:CmpLowInv}
and are calculated similarly by the formulas of theorem \ref{t:QdrLowInv}.
These cases have been completely unknown, up to now.

\end{example}

\newpage

\begin{table}[ht]
\caption{Second \(3\)-class groups with invariant \(\nu=0\) for \(D>0\)}
\label{tab:RealLowInv}
\begin{center}
\begin{tabular}{|lc|rrcr|rrrc|r|lr|}
\hline
 type & \(\varkappa\) & \(u\) & \(v\) & \(\mathrm{Cl}_3(\mathrm{F}_3^1(K))\) & \(w\) & \(e\) & \(m\) &  \(n\) & \(k\) &      min. \(D\) & \multicolumn{2}{|c|}{frequency} \\
\hline
 D.5  & \((4224)\)    & \(1\) & \(1\) &                          \((3,3,3)\) & \(3\) & \(3\) & \(4\) &  \(5\) & \(0\) &    \(631\,769\) &  &  \(47\) \\
 D.10 & \((2241)\)    & \(1\) & \(1\) &                          \((3,3,3)\) & \(3\) & \(3\) & \(4\) &  \(5\) & \(0\) &    \(422\,573\) &  &  \(93\) \\
\cline{12-13}
 E.6  & \((1313)\)    & \(2\) & \(1\) &                          \((9,9,3)\) & \(5\) & \(3\) & \(6\) &  \(7\) & \(0\) & \(5\,264\,069\) & \multirow{2}{*}{\(\Bigr\rbrace\)} & \multirow{2}{*}{\(7\)} \\
 E.14 & \((2313)\)    & \(2\) & \(1\) &                          \((9,9,3)\) & \(5\) & \(3\) & \(6\) &  \(7\) & \(0\) & \(3\,918\,837\) &&\\
\cline{12-13}
 E.8  & \((1231)\)    & \(2\) & \(1\) &                          \((9,9,3)\) & \(5\) & \(3\) & \(6\) &  \(7\) & \(0\) & \(6\,098\,360\) & \multirow{2}{*}{\(\Bigr\rbrace\)} & \multirow{2}{*}{\(14\)} \\
 E.9  & \((2231)\)    & \(2\) & \(1\) &                          \((9,9,3)\) & \(5\) & \(3\) & \(6\) &  \(7\) & \(0\) &    \(342\,664\) &&\\
\cline{12-13}
 F.13 & \((3143)\)    & \(2\) & \(2\) &                        \((9,9,9,3)\) & \(7\) & \(5\) & \(6\) &  \(9\) & \(0\) & \(8\,321\,505\) &  &   \(1\) \\
 G.16 & \((4231)\)    & \(2\) & \(1\) &                         \((27,9,3)\) & \(6\) & \(3\) & \(7\) &  \(8\) & \(1\) & \(8\,711\,453\) &  &   \(2\) \\
 G.19 & \((2143)\)    & \(1\) & \(1\) &                        \((3,3,3,3)\) & \(4\) & \(3\) & \(5\) &  \(6\) & \(1\) &    \(214\,712\) &  &  \(11\) \\
 H.4  & \((4443)\)    & \(1\) & \(1\) &                          \((9,3,3)\) & \(4\) & \(3\) & \(5\) &  \(6\) & \(1\) &    \(957\,013\) &  &  \(27\) \\
\hline
 F.13\(\uparrow\) & \((3143)\)    & \(3\) & \(2\) &                      \((27,27,9,3)\) & \(9\) & \(5\) & \(8\) & \(11\) & \(0\) & \(8\,127\,208\) &  &   \(1\) \\
 H.4\(\uparrow\)  & \((3313)\)    & \(2\) & \(1\) &                         \((27,9,3)\) & \(6\) & \(3\) & \(7\) &  \(8\) & \(1\) & \(1\,162\,949\) &  &   \(3\) \\
\hline
\multicolumn{11}{|r|}{total:}																																																				 	  &  & \(206\) \\
\hline
\end{tabular}
\end{center}
\end{table}

Among the \(2\,576\) real quadratic number fields \(K=\mathbb{Q}(\sqrt{D})\)
with discriminant \(D<10^7\)
and \(3\)-class group \(\mathrm{Cl}_3(K)\) of type \((3,3)\),
only a small part of \(67\) fields, that is \(2.6\%\),
has a single or double total principalisation in the dihedral fields \(N_1,\ldots,N_4\).
Consequently, the second \(3\)-class group
\(G=\mathrm{Gal}(\mathrm{F}_3^2(K)\vert K)\) is
a vertex on one of the coclass graphs \(\mathcal{G}(3,r)\)
with \(r\in\lbrace 2,3,4\rbrace\) \cite{EkLg,DEF}.

\begin{example}
\label{ex:RealLow}

For each principalisation type \(\varkappa\) with invariant \(1\le\nu\le 2\),
table \ref{tab:RealLow} shows
the minimal discriminant \(D\) and the absolute frequency
of real quadratic number fields \(K=\mathbb{Q}(\sqrt{D})\)
with \(3\)-class group of type \((3,3)\)
in the range \(0<D<10^7\) of discriminants,
whose second \(3\)-class group \(G=\mathrm{Gal}(\mathrm{F}_3^2(K)\vert K)\)
is of coclass \(\mathrm{cc}(G)\ge 2\)
with index of nilpotence \(m\) and invariants \(e\) and \(k\)
and realises the given type \(\varkappa\).
Every principalisation type \(\varkappa\) is characterised by
a lower case section letter \cite{Ne}
and an additional numerical identifier \cite{Ma3}.

\end{example}

\begin{table}[ht]
\caption{Second \(3\)-class groups with invariant \(1\le\nu\le 2\) for \(D>0\)}
\label{tab:RealLow}
\begin{center}
\begin{tabular}{|lcc|rrcr|rrrc|r|r|}
\hline
 type  & \(\varkappa\) & types of \(L_i\)               & \(u\) & \(v\) & \(\mathrm{Cl}_3(\mathrm{F}_3^1(K))\) & \(w\) & \(e\) & \(m\) &  \(n\) & \(k\) &      min. \(D\) & frequency \\
\hline
 b.10  & \((0043)\)    & \((\alpha\alpha\delta\delta)\) & \(2\) & \(2\) &                        \((9,9,3,3)\) & \(6\) & \(4\) & \(6\) &  \(8\) & \(1\) &    \(710\,652\) &  \(8\) \\
 c.18  & \((0313)\)    & \((\alpha\delta\delta\delta)\) & \(2\) & \(1\) &                          \((9,3,3)\) & \(4\) & \(3\) & \(5\) &  \(6\) & \(0\) &    \(534\,824\) & \(29\) \\
 c.21  & \((0231)\)    & \((\alpha\delta\delta\delta)\) & \(2\) & \(1\) &                          \((9,3,3)\) & \(4\) & \(3\) & \(5\) &  \(6\) & \(0\) &    \(540\,365\) & \(25\) \\
 d.19  & \((4043)\)    & \((\delta\alpha\delta\delta)\) & \(2\) & \(2\) &                        \((9,9,3,3)\) & \(6\) & \(4\) & \(6\) &  \(8\) & \(0\) & \(2\,328\,721\) &  \(1\) \\
 d.23  & \((1043)\)    & \((\delta\alpha\delta\delta)\) & \(2\) & \(2\) &                        \((9,9,3,3)\) & \(6\) & \(4\) & \(6\) &  \(8\) & \(0\) & \(1\,535\,117\) &  \(1\) \\
\hline
 c.21\(\uparrow\)  & \((0231)\)    & \((\alpha\delta\delta\delta)\) & \(3\) & \(1\) &                         \((27,9,3)\) & \(6\) & \(3\) & \(7\) &  \(8\) & \(0\) & \(1\,001\,957\) &  \(2\) \\
 d.25* & \((0143)\)    & \((\alpha\delta\delta\delta)\) & \(3\) & \(2\) &                       \((27,9,9,3)\) & \(8\) & \(5\) & \(7\) & \(10\) & \(0\) & \(8\,491\,713\) &  \(1\) \\
\hline
\multicolumn{12}{|r|}{total:}																																																				 	                                    & \(67\) \\
\hline
\end{tabular}
\end{center}
\end{table}

We start with the \(3\)-exponents \(u,v,w\) of the \(3\)-class numbers
\(\mathrm{h}_3(L_1)=3^u\), \(\mathrm{h}_3(L_2)=3^v\),
and \(\mathrm{h}_3(\mathrm{F}_3^1(K))=3^w\),
which have been computed together with the class group structures
as experimental data by means of PARI/GP \cite{PARI},
and calculate the index of nilpotence \(m\),
the \(3\)-exponent \(n\) of the group order \(\lvert G\rvert=3^n\),
and the invariants \(e\) and \(k\)
of the second \(3\)-class group \(G\) of \(K\)
by the formulas of theorem \ref{t:RealLowInv}:\\
\(m=2u+k+1\), \(n=2u+2v+k-1\), \(e=2v\), \(k=w-2u-2v+3\),
if \(T=(\alpha,\alpha)\) for the type b,\\
\(m=2u+1\), \(n=2u+2v\), \(e=2v+1\), \(k=w-2u-2v+2=0\),
if \(T=(\alpha,\delta)\) for types c and d*,\\
\(m=2u+2\), \(n=2u+2v\), \(e=2v\), \(k=w-2u-2v+2=0\),
if \(T=(\delta,\alpha)\) for the type d.

Each of these principalisation types,
which were completely unknown,
can occur with different indices of nilpotence \(m\),
the types of sections b,d,d* even with different values of the invariant \(e\).

Generally, in the representation of all metabelian \(3\)-groups \(G\)
with \(G/\gamma_2(G)\) of type \((3,3)\) as vertices in a directed tree
with root \(\mathrm{C}(3)\times\mathrm{C}(3)\) \cite[p.181 ff]{Ne},
the groups with principalisation types c and d*
are represented by \textit{capable} vertices on infinite \textit{main lines},
but all the other groups by \textit{terminal} vertices \cite{EkLg,DEF,EkFs}.
A criterion for separating types d and d* is given in \cite[Th.3.4]{Ma3}.


\section{Second \(p\)-class groups of quadratic fields of type \((p,p)\) with \(p\ge 5\)}
\label{s:InvThmQnt}

The increasing number of possible values of the invariant \(k\)
of metabelian \(p\)-groups \(G\) which are
vertices on the coclass \(1\) graphs \(\mathcal{G}(p,1)\)
for primes \(p\ge 5\) \cite{EkLg,DEF} 
makes it difficult to formulate useful inverse theorems.

For \(p=3\) with only two possibilities \(k\in\lbrace 0,1\rbrace\),
we had to compute the not quite easy \(3\)-class number of
the first Hilbert \(3\)-class field \(\mathrm{F}_3^1(K)\)
of a quadratic base field \(K=\mathbb{Q}(\sqrt{D})\)
with \(3\)-class group of type \((3,3)\),
that is a number field of absolute degree \(18\),
to determine the structure of
the second \(3\)-class group \(G=\mathrm{Gal}(\mathrm{F}_3^2(K)\vert K)\)
in section \ref{s:InvThmCub}.

On principle, the same method can also be applied to \(p\ge 5\).
However, the first Hilbert \(p\)-class field \(\mathrm{F}_p^1(K)\)
of a quadratic field \(K\) with \(p\)-class group of type \((p,p)\)
is a number field of absolute degree \(2p^2\ge 50\) in this case.

\begin{theorem}
\label{t:MaxInvQnt}

Let \(p\ge 5\) be an odd prime and
\(K\) a real quadratic field
with \(p\)-class group of type \((p,p)\).
Suppose that at least \(p\)
of the \(p+1\) totally real non-Galois subfields \(L_1,\ldots,L_{p+1}\) of absolute degree \(p\)
of the absolutely dihedral unramified field extensions \(N_1,\ldots,N_{p+1}\) of relative degree \(p\) of \(K\)
are of type \(\alpha\), say \(L_2,\ldots,L_{p+1}\)
with \(p\)-class numbers \(\mathrm{h}_p(L_2)=\ldots=\mathrm{h}_p(L_{p+1})=p\).
Assume that the remaining field \(L_1\) of absolute degree \(p\)
has the \(p\)-class number \(\mathrm{h}_p(L_1)=p^u\) with exponent \(u\ge 1\).
If the second \(p\)-class group
\(G=\mathrm{Gal}(\mathrm{F}_p^2(K)\vert K)\) of \(K\)
is of coclass \(\mathrm{cc}(G)=1\) with \(n=m\) and \(e=2\),
then the invariants \(m\) and \(k\) are given by
\[
\begin{cases}
m=2u+k+1\ge k+3,\ k\ge 0,&\text{ if }L_1\text{ is of type }\alpha,\\
m=2u+2\ge 4,\ k=0,&\text{ if }L_1\text{ is of type }\delta.\\
\end{cases}
\]
In the first case, the invariant \(k=w-2u+1\) is determined
by the \(p\)-class number \(\mathrm{h}_p(\mathrm{F}_p^1(K))=p^{w}\)
of the first Hilbert \(p\)-class field of \(K\),
and \(k\ge 1\) enforces \(u\ge 2\).

\end{theorem}

\begin{proof}

Since it is unknown,
which values the invariant \(\nu\) can take
for a second \(p\)-class group \(G\) of coclass \(\mathrm{cc}(G)\ge 2\)
in the case \(p\ge 5\),
the assumption \(p\le\nu\le p+1\) does not imply
that \(G\) is of coclass \(\mathrm{cc}(G)=1\).
Thus we have to assume explicitly that
\(G\) is of coclass \(\mathrm{cc}(G)=1\)
in the present proof.

By theorem \ref{t:QdrMaxDir}, we have\\
\(p^u=\mathrm{h}_p(L_1)=p^{\frac{m-k-1}{2}}\) with \(m-k\equiv 1(2)\),
for \(L_1\) of type \(\alpha\), \(k\ge 0\), but\\
\(p^u=\mathrm{h}_p(L_1)=p^{\frac{m-2}{2}}\) with \(m\equiv 0(2)\),
for \(L_1\) of type \(\delta\), \(k=0\).

For the group \(G\), we therefore obtain\\
an index of nilpotence \(m=2u+k+1\ge k+3\) with parity depending on \(k\),
if \(L_1\) is of type \(\alpha\) and \(k\ge 0\), and\\
an even index of nilpotence \(m=2u+2\ge 4\),
if \(L_1\) is of type \(\delta\) and \(k=0\).\\

For \(L_1\) of type \(\alpha\),
the invariant \(k\) is finally determined by
\(p^w=\mathrm{h}_p(\mathrm{F}_p^1(K))=\lvert\gamma_2(G)\rvert=p^{m-2}=p^{2u+k-1}\),
and thus \(k=w-2u+1\).
\end{proof}


\section{Second \(2\)-class groups of arbitrary fields of type \((2,2)\)}
\label{s:InvThmQdr}

Different from \(p\ge 3\), the value \(k=0\) is uniquely determined for \(p=2\).
Although the theory of dihedral fields cannot be applied to \(p=2\),
we get a useful criterion by inversion of the initial theorem \ref{t:MaxDir},
since the \(2\)-class numbers \(\mathrm{h}_2(N_i)\)
of the three unramified relatively quadratic extension fields \(N_1,\ldots,N_3\)
of an arbitrary base field \(K\) with \(2\)-class group of type \((2,2)\)
determine the order \(\lvert G\rvert=2^m\) and class \(\mathrm{cl}(G)=m-1\)
of the second \(2\)-class group
\(G=\mathrm{Gal}(\mathrm{F}_2^2(K)\vert K)\) of \(K\),
which is always of maximal class with \(m\ge 2\), that is,
a vertex on the coclass \(1\) graph \(\mathcal{G}(2,1)\)
\cite{EkLg,DEF,EkFs}. 

\begin{theorem}
\label{t:MaxInv}

Let \(K\) be an arbitrary base field
with elementary abelian bicyclic \(2\)-class group
and \(N_1,\ldots,N_3\) its
three unramified relatively quadratic extension fields.

\begin{enumerate}

\item
If the \(2\)-class numbers of \(N_1,\ldots,N_3\)
are given by \(\mathrm{h}_2(N_1)=\mathrm{h}_2(N_2)=\mathrm{h}_2(N_3)=2\),
then the second \(2\)-class group
\(G=\mathrm{Gal}(\mathrm{F}_2^2(K)\vert K)\) of \(K\)
is abelian of type \((2,2)\)
and thus of order \(\lvert G\rvert=2^m\) with \(m=2\).

\item
If at least two of the
three unramified relatively quadratic extension fields of \(K\),
say \(N_2,N_3\),
have \(2\)-class numbers \(\mathrm{h}_2(N_2)=\mathrm{h}_2(N_3)=2^2\)
and if the remaining relatively quadratic field \(N_1\)
has \(2\)-class number \(\mathrm{h}_2(N_1)=2^w\) with exponent \(w\ge 2\),
then the second \(2\)-class group
\(G=\mathrm{Gal}(\mathrm{F}_2^2(K)\vert K)\) of \(K\)
is metabelian of coclass \(\mathrm{cc}(G)=1\) and
of order \(\lvert G\rvert=2^m\) with exponent \(m=w+1\ge 3\).\\
Further, the \(2\)-class number of the first Hilbert \(2\)-class field of \(K\)
is given by \(\mathrm{h}_2(\mathrm{F}_2^1(K))=2^{w-1}\).

\end{enumerate}

\end{theorem}

\begin{proof}

By theorem \ref{t:MaxDir}, the following statements
for the index of nilpotence \(m\) of the \(2\)-group \(G\) hold.

\begin{enumerate}

\item
\(G\) is abelian with \(m=2\), if and only if
\(\mathrm{h}_2(N_1)=\mathrm{h}_2(N_2)=\mathrm{h}_2(N_3)=2\).

\item
\(G\) is metabelian of coclass \(\mathrm{cc}(G)=1\) with \(m\ge 3\), if and only if
\(\mathrm{h}_2(N_2)=\mathrm{h}_2(N_3)=2^2\) and \(\mathrm{h}_2(N_1)=2^{m-1}\).
Provided that \(\mathrm{h}_2(N_1)=2^w\) with \(w\ge 2\),
it follows that \(\lvert G\rvert=2^m\) with \(m=w+1\ge 3\) and
\(\mathrm{h}_2(\mathrm{F}_2^1(K))=\lvert\gamma_2(G)\rvert=2^{m-2}\)
with \(m-2=w-1\ge 1\).
\end{enumerate}
\end{proof}


\section{Computational results for complex quadratic fields of type \((2,2)\)}
\label{s:NumRstQdr}


\begin{example}
\label{ex:InvCmpQdrAbl}

For primes
\(p\equiv 3\pmod{8}\),
\(q\equiv 3\pmod{4}\)
with
\(\left(\frac{p}{q}\right)=-1\)
let
\(K=\mathbb{Q}(\sqrt{D})\)
be the complex quadratic base field
with discriminant \(D=-4pq\).
According to H. Kisilevsky \cite[p.277,(ii)]{Ki}
the \(2\)-class group of \(K\) is elementary abelian bicyclic of type \((2,2)\)
and coincides with the second \(2\)-class group
\(G=\mathrm{Gal}(\mathrm{F}_2^2(K)|K)\),
since this is a case of a single-stage tower
with \(\mathrm{F}_2^2(K)=\mathrm{F}_2^1(K)\).

\(K\) has three unramified quadratic relative extensions
\(N_1=K(\sqrt{-4})\), \(N_2=K(\sqrt{p})\), \(N_3=K(\sqrt{q})\),
sharing the same discriminant \(D^2\),
and a first Hilbert  \(2\)-class field
\(\mathrm{F}_2^1(K)=N_1\cdot N_2\)
with discriminant \(D^4\).
From the order \(2\) of the three cyclic \(2\)-class groups \(\mathrm{Cl}_2(N_i)\) with \(1\le i\le 3\)
we get the order \(\lvert G\rvert=2^2\) of \(G\), by theorem \ref{t:MaxInv}.
Table \ref{tab:InvCmpQdrAbl} shows the begin of the series with \(p=3\)
and \(q=7\).
The principalisation type was introduced in \cite[2.5]{Ma3}.

\end{example}

\renewcommand{\arraystretch}{1.2}
\begin{table}[ht]
\caption{Elementary abelian bicyclic second \(2\)-class group}
\label{tab:InvCmpQdrAbl}
\begin{center}
\begin{tabular}{|rcr|ccl|rrr|r|}
\hline
 \(D=-4pq\) & \(p\) &   \(q\) & \(m\) &                 \(G\)                & type & \(\mathrm{Cl}_2(N_1)\) & \(\mathrm{Cl}_2(N_2)\) & \(\mathrm{Cl}_2(N_3)\) & \(\mathrm{Cl}_2(\mathrm{F}_2^1(K))\) \\
\hline
    \(-84\) & \(3\) &   \(7\) & \(2\) &\(\mathrm{C}(2)\times \mathrm{C}(2)\) & a.1  &                \((2)\) &              \((2)\) &              \((2)\) &                              \(1\) \\
\hline
\end{tabular}
\end{center}
\end{table}


\begin{example}
\label{ex:InvCmpQdrDih}

For primes
\(p\equiv 3\pmod{8}\),
\(q\equiv 1\pmod{8}\)
with
\(\left(\frac{q}{p}\right)=-1\)
let
\(K=\mathbb{Q}(\sqrt{D})\)
be the complex quadratic base field
with discriminant \(D=-8pq\)
and \(\varepsilon\) the fundamental unit
of the real quadratic field \(k=\mathbb{Q}(\sqrt{8q})\)
with discriminant \(8q\).
Then, according to Kisilevsky \cite[p.278,(vi)(b)]{Ki},
the \(2\)-class group of \(K\) is of type \((2,2)\)
and the second \(2\)-class group
\(G=\mathrm{Gal}(\mathrm{F}_2^2(K)|K)\)
is isomorphic either to a dihedral group \(D(2^m)\) with \(m\ge 3\),
if \(\mathrm{Norm}_{k|\mathbb{Q}}(\varepsilon)=+1\),
or to a generalised quaternion group \(Q(2^m)\) with \(m\ge 4\),
if \(\mathrm{Norm}_{k|\mathbb{Q}}(\varepsilon)=-1\).

\(K\) has three unramified quadratic relative extensions
\(N_1=K(\sqrt{-p})\), \(N_2=K(\sqrt{q})\), \(N_3=K(\sqrt{8})\),
sharing the same discriminant \(D^2\),
and a first Hilbert \(2\)-class field
\(\mathrm{F}_2^1(K)=N_1\cdot N_2\)
with discriminant \(D^4\).
By theorem \ref{t:MaxInv},
the order \(2^{m-1}\) of the cyclic \(2\)-class group \(\mathrm{Cl}_2(N_1)\)
determines the order \(\lvert G\rvert=2^m\) of \(G\).
Table \ref{tab:InvCmpQdrDih} shows the begin of the series with \(p=3\)
and increasing values of the parameter \(q\),
using the notation of \cite[Th.2.6]{Ma3}.

\end{example}

\begin{table}[ht]
\caption{Dihedral and quaternion groups of increasing order as second \(2\)-class groups}
\label{tab:InvCmpQdrDih}
\begin{center}
\begin{tabular}{|rcr|ccl|rrr|r|}
\hline
 \(D=-8pq\) & \(p\) &   \(q\) & \(m\) &                          \(G\) & type & \(\mathrm{Cl}_2(N_1)\) & \(\mathrm{Cl}_2(N_2)\) & \(\mathrm{Cl}_2(N_3)\) & \(\mathrm{Cl}_2(\mathrm{F}_2^1(K))\) \\
\hline
   \(-408\) & \(3\) &  \(17\) & \(3\) &  \(D(8)\simeq G^{(3)}_0(0,0)\) & d.8  &                \((4)\) &              \((2,2)\) &              \((2,2)\) &                              \((2)\) \\
  \(-6168\) & \(3\) & \(257\) & \(4\) & \(D(16)\simeq G^{(4)}_0(0,0)\) & d.8\(\uparrow\)  &                \((8)\) &              \((2,2)\) &              \((2,2)\) &                              \((4)\) \\
 \(-29208\) & \(3\) &\(1217\) & \(5\) & \(D(32)\simeq G^{(5)}_0(0,0)\) & d.8\(\uparrow^2\)  &               \((16)\) &              \((2,2)\) &              \((2,2)\) &                              \((8)\) \\
 \(-609816\) & \(3\) &\(25409\) & \(6\) & \(D(64)\simeq G^{(6)}_0(0,0)\) & d.8\(\uparrow^3\)  &               \((32)\) &              \((2,2)\) &              \((2,2)\) &                              \((16)\) \\
 \(-670872\) & \(3\) &\(27953\) & \(7\) & \(D(128)\simeq G^{(7)}_0(0,0)\) & d.8\(\uparrow^4\)  &               \((64)\) &              \((2,2)\) &              \((2,2)\) &                              \((32)\) \\
\hline
   \(-984\) & \(3\) &  \(41\) & \(4\) & \(Q(16)\simeq G^{(4)}_0(0,1)\) & Q.6  &                \((8)\) &              \((2,2)\) &              \((2,2)\) &                              \((4)\) \\
  \(-2712\) & \(3\) & \(113\) & \(5\) & \(Q(32)\simeq G^{(5)}_0(0,1)\) & Q.6\(\uparrow\)  &               \((16)\) &              \((2,2)\) &              \((2,2)\) &                              \((8)\) \\
\hline
\end{tabular}
\end{center}
\end{table}


\begin{example}
\label{ex:InvCmpQdrSem}

For primes
\(p\equiv 1\pmod{4}\),
\(q\equiv 1\pmod{4}\)
with
\(pq\equiv 5\pmod{8}\)
and
\(\left(\frac{p}{q}\right)=-1\)
let
\(K=\mathbb{Q}(\sqrt{D})\)
be the complex quadratic base field
with discriminant \(D=-4pq\).
According to Kisilevsky \cite[p.277,(i)]{Ki},
the \(2\)-class group of \(K\) is of type \((2,2)\)
and the second \(2\)-class group
\(G=\mathrm{Gal}(\mathrm{F}_2^2(K)|K)\)
is isomorphic to a semidihedral group \(S(2^m)\) with \(m\ge 4\).

\(K\) has three unramified quadratic relative extensions
\(N_1=K(\sqrt{p})\), \(N_2=K(\sqrt{q})\), \(N_3=K(\sqrt{-4})\)
sharing the same discriminant \(D^2\),
and a first Hilbert \(2\)-class field
\(\mathrm{F}_2^1(K)=N_1\cdot N_2\)
with discriminant \(D^4\).
By theorem \ref{t:MaxInv},
the order \(2^{m-1}\) of the cyclic \(2\)-class group \(\mathrm{Cl}_2(N_1)\)
determines the order \(\lvert G\rvert=2^m\) of \(G\).
Table \ref{tab:InvCmpQdrSem} shows the begin of the series with \(p=5\)
and increasing values of the parameter \(q\),
using the notation of \cite[Th.2.6]{Ma3}.

\end{example}

\begin{table}[ht]
\caption{Semidihedral groups of increasing order as second \(2\)-class groups}
\label{tab:InvCmpQdrSem}
\begin{center}
\begin{tabular}{|rcr|ccl|rrr|r|}
\hline
 \(D=-4pq\) & \(p\) &   \(q\) & \(m\) &                           \(G\) & type & \(\mathrm{Cl}_2(N_1)\) & \(\mathrm{Cl}_2(N_2)\) & \(\mathrm{Cl}_2(N_3)\) & \(\mathrm{Cl}_2(\mathrm{F}_2^1(K))\) \\
\hline
   \(-340\) & \(5\) &  \(17\) & \(4\) &  \(S(16)\simeq G^{(4)}_0(1,0)\) & S.4  &                \((8)\) &              \((2,2)\) &              \((2,2)\) &                              \((4)\) \\
  \(-2260\) & \(5\) & \(113\) & \(5\) &  \(S(32)\simeq G^{(5)}_0(1,0)\) & S.4\(\uparrow\)  &               \((16)\) &              \((2,2)\) &              \((2,2)\) &                              \((8)\) \\
  \(-5140\) & \(5\) & \(257\) & \(6\) &  \(S(64)\simeq G^{(6)}_0(1,0)\) & S.4\(\uparrow^2\)  &               \((32)\) &              \((2,2)\) &              \((2,2)\) &                             \((16)\) \\
 \(-17140\) & \(5\) & \(857\) & \(7\) & \(S(128)\simeq G^{(7)}_0(1,0)\) & S.4\(\uparrow^3\)  &               \((64)\) &              \((2,2)\) &              \((2,2)\) &                             \((32)\) \\
 \(-165460\) & \(5\) &\(8273\) & \(8\) & \(S(256)\simeq G^{(8)}_0(1,0)\) & S.4\(\uparrow^4\)  &              \((128)\) &              \((2,2)\) &              \((2,2)\) &                             \((64)\) \\
\hline
\end{tabular}
\end{center}
\end{table}


\begin{example}
\label{ex:InvCmpQdrQtr}

For primes
\(p\equiv 3\pmod{8}\),
\(q\equiv 5\pmod{8}\)
let
\(K=\mathbb{Q}(\sqrt{D})\)
be the complex quadratic base field
with discriminant \(D=-8pq\).
Then,
according to Kisilevsky \cite[p.278,(vi)(a,b)]{Ki},
the \(2\)-class group of \(K\) is of type \((2,2)\)
and the second \(2\)-class group
\(G=\mathrm{Gal}(\mathrm{F}_2^2(K)|K)\)
is isomorphic either to the quaternion group \(Q(2^3)\) of order eight, if \(\left(\frac{q}{p}\right)=-1\),
or to a generalised quaternion group \(Q(2^m)\) with \(m\ge 4\), if \(\left(\frac{q}{p}\right)=+1\).

\(K\) has three unramified quadratic relative extensions
\(N_1=K(\sqrt{8})\), \(N_2=K(\sqrt{q})\), \(N_3=K(\sqrt{-p})\)
sharing the same discriminant \(D^2\),
and a first Hilbert \(2\)-class field
\(\mathrm{F}_2^1(K)=N_1\cdot N_2\)
with discriminant \(D^4\).
The order \(2^{m-1}\) of the cyclic \(2\)-class group \(\mathrm{Cl}_2(N_1)\)
determines the order \(\lvert G\rvert=2^m\) of \(G\), by theorem \ref{t:MaxInv}.
Table \ref{tab:InvCmpQdrQtr} shows the begin of the series with \(p=3\)
and increasing values of the parameter \(q\),
using the notation of \cite[Th.2.6]{Ma3}.

\end{example}

\newpage

\begin{table}[ht]
\caption{Quaternion groups of increasing order as second \(2\)-class groups}
\label{tab:InvCmpQdrQtr}
\begin{center}
\begin{tabular}{|rcr|ccl|rrr|r|}
\hline
 \(D=-8pq\) & \(p\) &   \(q\) & \(m\) &                           \(G\) & type & \(\mathrm{Cl}_2(N_1)\) & \(\mathrm{Cl}_2(N_2)\) & \(\mathrm{Cl}_2(N_3)\) & \(\mathrm{Cl}_2(\mathrm{F}_2^1(K))\) \\
\hline
   \(-120\) & \(3\) &   \(5\) & \(3\) &   \(Q(8)\simeq G^{(3)}_0(0,1)\) & Q.5  &                \((4)\) &                \((4)\) &                \((4)\) &                              \((2)\) \\
\hline
   \(-312\) & \(3\) &  \(13\) & \(4\) &  \(Q(16)\simeq G^{(4)}_0(0,1)\) & Q.6  &                \((8)\) &              \((2,2)\) &              \((2,2)\) &                              \((4)\) \\
   \(-888\) & \(3\) &  \(37\) & \(5\) &  \(Q(32)\simeq G^{(5)}_0(0,1)\) & Q.6\(\uparrow\)  &               \((16)\) &              \((2,2)\) &              \((2,2)\) &                              \((8)\) \\
  \(-3768\) & \(3\) & \(157\) & \(6\) &  \(Q(64)\simeq G^{(6)}_0(0,1)\) & Q.6\(\uparrow^2\)  &               \((32)\) &              \((2,2)\) &              \((2,2)\) &                             \((16)\) \\
  \(-8952\) & \(3\) & \(373\) & \(7\) & \(Q(128)\simeq G^{(7)}_0(0,1)\) & Q.6\(\uparrow^3\)  &               \((64)\) &              \((2,2)\) &              \((2,2)\) &                             \((32)\) \\
 \(-40632\) & \(3\) &\(1693\) & \(8\) & \(Q(256)\simeq G^{(8)}_0(0,1)\) & Q.6\(\uparrow^4\)  &              \((128)\) &              \((2,2)\) &              \((2,2)\) &                             \((64)\) \\
\hline
\end{tabular}
\end{center}
\end{table}


\section{Final remarks}
\label{s:Acknowledgements}

All numerical results of the sections \ref{s:NumRstCub} and \ref{s:NumRstQdr}
have been calculated with the aid of programs which we have developed
for the number theoretic computer algebra system PARI/GP, version 2.3.4 (2008)
\cite{Be,PARI}.
Details of our algorithms are presented in the related paper \cite{Ma4}.

The examples \ref{ex:InvCmpQdrDih}, \ref{ex:InvCmpQdrSem}, and \ref{ex:InvCmpQdrQtr}
suggest the conjecture that
dihedral, semidihedral, and quaternion groups of arbitrarily high order are to be expected
as second \(2\)-class groups \(G=\mathrm{Gal}(\mathrm{F}_2^2(K)\vert K)\)
in the further continuation of the investigated series of complex quadratic base fields \(K\).



\end{document}